\documentclass[11pt,a4paper]{article}
\date{}
\makeatletter
\@addtoreset{equation}{section}
\makeatother
\usepackage[colorlinks,
            linkcolor=blue,
            anchorcolor=blue,
            citecolor=blue]{hyperref}
\usepackage{hyperref}
\hypersetup{hypertex=true,
            colorlinks=true,
            linkcolor=blue,
            anchorcolor=blue,
            citecolor=blue}
\usepackage{hyperref,amsmath,amssymb,amscd}
\usepackage{amssymb,amsmath,enumerate}
\usepackage{indentfirst}
\usepackage{url}
\usepackage[capitalise]{cleveref}
\usepackage{amsthm}
\usepackage{cite}
\newtheorem{thm}{Theorem}[section]

\newtheorem{lem}[thm]{Lemma}
\newtheorem{prop}[thm]{Proposition}
\theoremstyle{definition}

\newtheorem{rem}[thm]{\bf Remark}
\allowdisplaybreaks[4]

\begin{document}
\title{\bf
Mass concentration of minimizers for $L^2$-subcritical Kirchhoff energy functional in bounded domains \footnote{Supported by National Natural Science Foundation of China (No. 12371120).}}
\author{{Chen Yang,\ Shubin Yu,\ Chun-Lei Tang\footnote{Corresponding author.
E-mail address: yangchen6858@163.com (C. Yang), yshubin168@163.com (S. Yu), tangcl@swu.edu.cn (C.-L. Tang).}}\\
{\small \emph{ School of Mathematics and Statistics, Southwest University,  Chongqing {\rm400715},}}\\
{\small \emph{People's Republic of China}}}
\maketitle
\baselineskip 17pt

{\bf Abstract}: We are concerned with $L^2$-constraint minimizers for the Kirchhoff functional
$$
E_b(u)=\int_{\Omega}|\nabla u|^2\mathrm{d}x+\frac{b}{2}\left(\int_\Omega|\nabla u|^2\mathrm{d}x\right)^2+\int_\Omega V(x)u^2\mathrm{d}x-\frac{\beta}{2}\int_{\Omega}|u|^4\mathrm{d}x,
$$
where $b>0$, $\beta>0$ and $V(x)$ is a trapping potential in a bounded domain $\Omega$ of $\mathbb R^2$. As is well known that minimizers exist for any $b>0$ and $\beta>0$, while the minimizers do not exist for $b=0$ and $\beta\geq\beta^*$, where $\beta^*=\int_{\mathbb R^2}|Q|^2\mathrm{d}x$ and $Q$ is the unique positive solution of $-\Delta u+u-u^3=0$ in $\mathbb R^2$.
In this paper, we show that for $\beta=\beta^*$, the energy converges to 0, but for $\beta>\beta^*$, the minimal energy will diverge to $-\infty$ as $b\searrow0$.
Further, we give the refined limit behaviors and energy estimates of minimizers as $b\searrow0$ for $\beta=\beta^*$ or $\beta>\beta^*$.
For both cases, we obtain that the mass of minimizers concentrates either at an inner point or near the boundary of $\Omega$, depending on whether $V(x)$ attains its flattest global minimum at an inner point of
$\Omega$ or not.
Meanwhile, we find an interesting phenomenon that the blow-up rate when the minimizers concentrate near the boundary of $\Omega$ is faster than concentration at an interior point if $\beta=\beta^*$, but the blow-up rates remain consistent if $\beta>\beta^*$.



\section{Introduction and main results}\label{sec:introduction}

In this article, we mainly focus on the following
constraint minimization problem
\begin{equation}\label{eqn:this-paper-minimizing-problem}
e(b)=\inf_{\{u\in\mathcal H_0^1(\Omega):\int_{\Omega}u^2\mathrm{d}x=1\}} E_b(u),
\end{equation}
where $\Omega\subset\mathbb R^2$ is a bounded domain and $E_b(u)$ is a Kirchhoff type functional as follows
$$
E_b(u)=\int_{\Omega}|\nabla u|^2\mathrm{d}x+\frac{b}{2}\left(\int_\Omega|\nabla u|^2\mathrm{d}x\right)^2+\int_\Omega V(x)u^2\mathrm{d}x-\frac{\beta}{2}\int_{\Omega}|u|^4\mathrm{d}x.
$$
Moreover, we assume that $b>0$, $\beta>0$ and $V(x)$ is a trapping potential satisfying
\begin{itemize}
  \item [($V_1$)] $\min_{x\in\bar{\Omega}}V(x)=0$\ and\ $V(x)\in C^\alpha(\bar{\Omega})$\ for\ some\ $0<\alpha<1$.
\end{itemize}
Here $\Omega\subset \mathbb R^2$ has $C^1$ boundary and satisfies the interior ball condition in the sense that for all
$x_0\in\partial\Omega$, there exists an open ball $B\subset\Omega$ such that $x_0\in\partial B\cap\partial\Omega$.
The minimizers of $e(b)$ correspond to ground states of the following Kirchhoff equation with some Lagrange multiplier $\mu\in \mathbb R$
$$
\left\{\begin{array}{ll}
-(1+b\int_{\Omega}|\nabla u|^2\mathrm{d}x)\Delta u+V(x)u=\mu u+\beta u^3  &\mbox{in}\ {\Omega}, \\[0.1cm]
 u=0&\mbox{on}\ {\partial\Omega}, \\[0.1cm]
\end{array}
\right.
$$
which is a steady-state equation of certain type of generalized Kirchhoff equation.
The Kirchhoff type equation was first proposed by Kirchhoff \cite{Kirchhoff1883} as an extension of the classical D'Alembert's wave equations for the transversal oscillations of a stretched string. We refer the reader to \cite{Arosio1996, Cavalcanti2001} and references therein for more physical background.

As is known that $E_b(u)$ is a $L^2$-subcritical functional and the minimizers of $e(b)$ always exist for any $\beta\geq0$ if $b>0$. This fact directly follows the classical Gagliardo-Nirenberg inequality in bounded domains, see for example \cite[Appendix A]{Noris-2014},
\begin{equation}\label{eqn:GN-type-inequality}
\int_{\Omega}|u|^{4}\mathrm{d}x\leq \frac{2}{\beta^*}\int_\Omega|\nabla u|^2\mathrm{d}x
\int_\Omega |u|^2\mathrm{d}x,\ u\in H^1_0(\Omega),
\end{equation}
where $\beta^*=\int_{\mathbb R^2}|Q|^2\mathrm{d}x$
and $Q$ is the unique positive radially symmetric solution of the following nonlinear scalar field equation
\begin{equation}\label{eqn:Q-unique-radial}
-\Delta u+u-u^3=0,\ u\in H^1(\mathbb R^2).
\end{equation}
However, following \cite[Theorem 1.2]{Luo-2019}, the minimizers do not exist for $\beta\geq\beta^*$ if $b=0$.
Thus, it is natural to ask what happens
to the minimizers of $e(b)$ with $\beta\geq\beta^*$ as $b\searrow0$.
It is the purpose of this paper to show that the mass of minimizers
concentrates at a flattest global minimum of the classical trapping potential $V(x)$ (see (\ref{eqn:V-potential-x-x-i}) below), and the global minimum point may be within or at the boundary since $\Omega$ is a bounded domain. This concentration phenomenon was discovered in \cite{Guo-Luo-Zhang-2018}, where the authors considered the $L^2$-critical Hartree energy functional, but it seems unknown whether it can occur for the above $L^2$-subcritical Kirchhoff functional in bounded domains.

When $b=0$, the above $E_b(u)$ becomes the famous Gross-Pitaevskii (GP) energy functional,
$$
E_0(u)=\int_{\Omega}|\nabla u|^2\mathrm{d}x+\int_\Omega V(x)u^2\mathrm{d}x-\frac{\beta}{2}\int_{\Omega}|u|^4\mathrm{d}x,
$$
which can describe the trapped Bose-Einstein condensates with attractive interactions, see \cite{Dalfovo-1999,Pitaevskii-1961,Gross-1963}. Recently, the authors in \cite{Yang-2024} considered the minimizing problem
$$
e(0)=\inf\limits_{\{u\in\mathcal H_0^1(\Omega):\int_{\Omega}u^2\mathrm{d}x=1\}}E_0(u).
$$
Further, assume that the trapping potential $V$ has $n\geq1$ isolated
minima, and that in their vicinity $V$ behaves like a power of the distance from these points, i.e.,
\begin{equation}\label{eqn:V-potential-x-x-i}
V(x)=h(x)\prod_{i=1}^n|x-x_i|^{p_i}\ \mbox{with}\ C<h(x)<\frac{1}{C}\ \mbox{for\ all}\ x\in\bar\Omega,
\end{equation}
where $p_i>0$, $C>0$, $\lim_{x\rightarrow x_i}h(x)$ exists and $x_i\not=x_j$
for all $1\leq i, j\leq n$ with $i\not=j$. They proved that if $V(x)$ attains its flattest global minimum only at the boundary of $\Omega$, the mass of minimizers must concentrate near the boundary of $\Omega$ as $\beta\nearrow\beta^*$. Before that, Luo and Zhu in \cite{Luo-2019} investigated the limit behavior of minimizers for $e(0)$ as $\beta\nearrow\beta^*$ and proved that the mass concentrates at an inner point of $\Omega$ if one of flattest global minimum points is within the domain $\Omega$. For more results involving GP functional, we refer the reader to
\cite{Li-2021,Guo-2016-AIHP,Guo-2018-Nonlinearity,Guo-2017-SIAM,Guo-2014}.

In terms of the concentration phenomenon of minimizers for nonlocal Kirchhoff problems, we mention that it has been extensively studied if $\Omega=\mathbb R^N$, see for instance \cite{Hu-Lu-2023,Guo-2021,Hu-Tang-2021,Guo-2018-CPAA,Du-2024,Hu-Zhao-2025}. In particular, the authors in \cite{Guo-2021} investigated the following minimization problem
$$
\tilde e(b)=\inf_{\{u\in\mathcal H:\int_{\mathbb R^2}u^2\mathrm{d}x=1\}}\tilde E_b(u),
$$
where $\mathcal H:=\{u\in H^1(\mathbb R^2):\int_{\mathbb R^2}V(x)u^2\mathrm dx<\infty\}$ and the functional $\tilde E(b)$ is defined by
$$
\tilde E(b)=\int_{\mathbb R^2}|\nabla u|^2\mathrm{d}x+\frac{b}{2}\left(\int_{\mathbb R^2}|\nabla u|^2\mathrm{d}x\right)^2+\int_{\mathbb R^2} V(x)u^2\mathrm{d}x-\frac{\beta}{2}\int_{\mathbb R^2}|u|^4\mathrm{d}x.
$$
One can know that the minimizers exist for any $\beta\geq0$ and $b>0$, but the minimizers do not exist for $\beta\geq\beta^*$ and $b=0$. Then the mass concentration behavior of the minimizers of $\tilde e(b)$ for $\beta\geq\beta^*$ was investigated as $b\searrow0$.
When $\Omega$ is a bounded domain, only the authors in \cite{Yu-2024} discussed the Kirchhoff energy functional involving mass critical term
\begin{equation}\label{eqn:AML-Kirchhoff-energy-in-bounded-domains}
E_a(u):=a\int_{\Omega}|\nabla u|^2\mathrm{d}x+\frac{b}{2}\left(\int_\Omega|\nabla u|^2\mathrm{d}x\right)^2+\int_\Omega V(x)u^2\mathrm{d}x
-\frac{3\bar\beta}{7}\int_{\Omega}|u|^{2+\frac{8}{3}}\mathrm{d}x,
\end{equation}
where $a\geq0$, $b>0$, $\bar\beta$ is a given constant and $\Omega\subset\mathbb R^3$ is a bounded domain. Under the assumption $(V_1)$, they determined that $L^2$-constraint minimizers corresponding to functional \eqref{eqn:AML-Kirchhoff-energy-in-bounded-domains} exist if and only if $a>0$. Further suppose that $V(x)$ satisfies \eqref{eqn:V-potential-x-x-i}, they gave the fine limit profiles of minimizers, including mass concentration at an inner point or near the boundary of $\Omega$ as $a\searrow0$.
%Replacing $\Omega$ in \eqref{eqn:AML-Kirchhoff-energy-in-bounded-domains} by $\mathbb R^N$ and coefficient $\bar\beta=1$, $N=1,2,3$, we refer the reader to \cite{Hu-Tang-2021} for the limiting profiles of minimizers as $a\searrow0$, \cite{Hu-Lu-2023} for concentration of normalized ground states as $c\nearrow c_*$ with $a=a(x)\in C(\mathbb R^N)$ and $V(x)=0$.

Based on the above arguments, we are concerned with the mass concentration of minimizers for constrained minimization problem \eqref{eqn:this-paper-minimizing-problem}, and it can be seen that the problem is similar to \cite{Guo-2021}. Nevertheless, since $\Omega$ is a bounded domain, $V(x)$ may only vanish at the boundary and the concentration may not occur within the domain $\Omega$. In addition, the new energy property will appear for our problem with $\beta>\beta^*$, which is different from \cite{Yu-2024,Yang-2024,Luo-2019} mentioned above.
Hence, how to establish the concentration behaviors in this case is a new challenge.

%By classifying parameter $\beta$, we obtain the limit behavior for %both case $\beta=\beta^*$ and $\beta>\beta^*$ as $b\searrow0$, %respectively. For the former regime, whether the mass concentrates
%in the interior or at the boundary of $\Omega$ (it depends on whether $V(x)$ attains its flattest global minimum at an inner point or near the boundary of $\Omega$), energy estimate and the delicate limit behavior of minimizers as $b\searrow0$ are established.

In the following, we collect some existence results and establish  the properties of $e(b)$ for $\beta=\beta^*$ and $\beta>\beta^*$ as $b\searrow0$. For $\beta>\beta^*$, in order to present the results better, we introduce the constrained minimization problem
\begin{equation}\label{eqn:auxiliary-minimization-problem}
\bar e(b)=\inf_{\{u\in H^1(\mathbb R^2):\int_{\mathbb R^2}u^2\mathrm{d}x=1\}} \bar E_b(u),
\end{equation}
where
$$
\bar E_b(u)=\int_{\mathbb R^2}|\nabla u|^2\mathrm{d}x+\frac{b}{2}\left(\int_{\mathbb R^2}|\nabla u|^2\mathrm{d}x\right)^2-\frac{\beta}{2}\int_{\mathbb R^2}|u|^4\mathrm{d}x.
$$
The explicit form of $\bar e(b)$ is as follows, which implies that $\bar e(b)\to-\infty$ as $b\searrow0$.

\begin{prop}{\rm(\!\!\cite[Lemma 2.1]{Guo-2021})}\label{pro:bar-e-b-explicit-form-bar-u-b}
There hold that
\begin{equation}\label{eqn:bar-e-b-value}
\bar e(b)=-\frac{1}{2b}\Big(\frac{\beta-\beta^*}{\beta^*}\Big)^2
\end{equation}
and the unique (up to translations) nonnegative minimizer of $\bar e(b)$ is of the form
\begin{equation}\label{eqn:bar-u-b-form-r-b-form}
\bar u_b(x)=\frac{r_b^{\frac{1}{2}}}{\sqrt{\beta^*}}Q(r_b^{\frac{1}{2}}x),\
\mbox{where}\ r_b=\frac{\beta-\beta^*}{b\beta^*}.
\end{equation}
\end{prop}

\begin{thm}\label{thm:existence-nonexistence-e-0}
Assume that $V(x)$ satisfies $(V_1)$, then
\begin{itemize}\setlength{\itemsep}{-4pt}
  \item [($i$)] if $b=0$, there exists at least one nonnegative minimizer for $e(b)$ for $0\leq\beta<\beta^*=\int_{\mathbb R^2}Q^2\mathrm{d}x$ and there is no minimizer for $e(b)$ for $\beta\geq\beta^*$; if $b>0$, $e(b)$ has always a nonnegative minimizer for any $\beta\geq0$.
  \item [($ii$)] $e(0)=0$ for $\beta=\beta^*$ and $e(0)=-\infty$ for $\beta>\beta^*$.
  \item [($iii$)] $e(b)\to e(0)$ as $b\searrow0$ if $\beta=\beta^*$ and $e(b)-\bar e(b)\to0$ as $b\searrow0$ if $\beta>\beta^*$.
\end{itemize}
\end{thm}

\begin{rem}
As mentioned previous, $e(b)$ has a minimizer for any $\beta\geq0$ if $b>0$. Moreover, one can see that the minimizers must be either nonnegative or nonpositive. Without loss of generality, we focus on nonnegative minimizers of $e(b)$. The rest of $(i)$-($ii$) can be seen in \cite[Theorem 1.2]{Luo-2019}. Observe that $e(b)\to0$ as $b\searrow0$ for $\beta=\beta^*$, which is similar to  \cite{Yu-2024,Yang-2024,Luo-2019}. However, $e(b)\to-\infty$ as $b\searrow0$ if $\beta>\beta^*$ and we will provide the first contribution in this direction.
\end{rem}

Based on Theorem \ref{thm:existence-nonexistence-e-0}, we can analyse the behavior of nonnegative minimizers as $b\searrow0$ and know that the blow-up occurs, see Lemmas \ref{lem:blow-up-beta-equiv} and \ref{lem:nabla-u-b-r-b-mass-critical-term-estimate}. But the further blow-up behavior and concentration
phenomenon rely on the shape of $V(x)$.
To this end, we study the trapping potential \eqref{eqn:V-potential-x-x-i}. Inspired by \cite{Guo-Luo-Zhang-2018}, it requires to consider the locations of the flattest global minima of $V(x)$ and then we define
$$
\mathcal Z_1:=\{x_i\in\Omega:p_i=p\}\ \mbox{and}\ \mathcal Z_0:=\{x_i\in\partial\Omega:p_i=p\},
$$
where $p:=\max\{p_1,p_2,\cdots,p_n\}$.
Then it suffices to discuss
$\mathcal Z_1\not=\emptyset$
or not, and we will establish the refined energy estimates and limit profiles of minimizers as $b\searrow0$ for $\beta=\beta^*$ and $\beta>\beta^*$, respectively.
Furthermore, for simplicity, we denote
$$
\kappa_i:=\lim_{x\rightarrow x_i}\frac{V(x)}{|x-x_i|^p}\ \mbox{and}\ \kappa:=\min\{\kappa_1,\kappa_2,\cdots,\kappa_n\},
$$
$$
\lambda_i:=\left(\frac{p\kappa_i}{2\beta^*}\int_{\mathbb{R}^{2}}|x|^pQ^2(x)\mathrm{d}x\right)^{\frac{1}{p+2}}
\ \mbox{and}\ \lambda=\min\{\lambda_i:x_i\in\Omega,1\leq i\leq n\}.
$$

In the sequel, we fix $\beta=\beta^*$ and present the main results for the cases of $\mathcal Z_1\neq\emptyset$ or $\mathcal Z_1=\emptyset$.

\begin{thm}\label{thm:mass-concentration-minimizer-beta-inter}
Assume that $V(x)$ satisfies \eqref{eqn:V-potential-x-x-i} and $\mathcal Z_1\neq\emptyset$. Let $u_b$ be a nonnegative minimizer of $e(b)$ for $b>0$ and $\beta=\beta^*$. For any sequence $\{b_k\}$ satisfying $b_k\searrow0$ as $k\rightarrow\infty$, there exists a subsequence, still denoted by $\{b_k\}$, such that each $u_{b_k}$ has a unique maximum point $z_{b_k}$ and
\begin{equation}\label{eqn:beta-inner-u-b-k-thm}
\lim\limits_{k\rightarrow\infty}\epsilon_{b_k}u_{b_k}(\epsilon_{b_k}x+z_{b_k})=\frac{Q(|x|)}{\sqrt{\beta^*}}\ \mbox{in}\ H^1(\mathbb R^2).
\end{equation}
where $\lim_{k\rightarrow\infty}z_{b_k}=x_0\in\mathcal Z_1$, $Q(x)$ is the unique positive radially symmetric solution of \eqref{eqn:Q-unique-radial}.
Moreover,
$$
\lim\limits_{k\to\infty}\frac{e(b_k)}{b_k^{\frac{p}{p+4}}}=\frac{p+2}{2p}\lambda^{\frac{4(p+2)}{p+4}},
$$
$$
\lim\limits_{k\to\infty}\frac{\epsilon_{b_k}}{b_k^{\frac{1}{p+4}}}=\lambda^{-\frac{p+2}{p+4}}
\ \mbox{and}\
\lim_{k\rightarrow\infty}\frac{|z_{b_k}-x_0|}
{\epsilon_{b_k}}=0.
$$
\end{thm}

\begin{thm}\label{thm:mass-concentration-minimizer-beta-boundary}
Suppose that $V(x)$ satisfies \eqref{eqn:V-potential-x-x-i} and $\mathcal Z_1=\emptyset$. Let $u_b$ be a nonnegative minimizer of $e(b)$ for $b>0$ and $\beta=\beta^*$. For any sequence $\{b_k\}$ satisfying $b_k\searrow0$ as $k\rightarrow\infty$, there exists a subsequence, still denoted by $\{b_k\}$, such that $u_{b_k}$ also satisfies \eqref{eqn:beta-inner-u-b-k-thm} and the unique maximum point
$z_{b_k}\to x_0\in\mathcal Z_0$ as $k\to\infty$.
Moreover,
\begin{equation}\label{eqn:energy-estimates-ln}
\lim\limits_{k\to\infty}\frac{e(b_k)}{b_k^{\frac{p}{p+4}}\Big(\ln\frac{2}{b_k}\Big)^{\frac{4p}{p+4}}}=\kappa^{\frac{4}{p+4}}2^{-\frac{5p}{p+4}}
\Big(\frac{p+2}{p+4}\Big)^{\frac{4p}{p+4}}
\Big[\Big(\frac{p}{4}\Big)^{\frac{4}{p+4}}+\Big(\frac{4}{p}\Big)^{\frac{p}{p+4}}\Big],
\end{equation}
$$
\lim\limits_{k\to\infty}
\frac{\epsilon_{b_k}}{b_k^{\frac{1}{p+4}}\Big(\ln\frac{2}{b_k}\Big)^{\frac{p}{p+4}}}
=\Big(\frac{2^{p+1}}{p\kappa}\Big)^{\frac{1}{p+4}}
\Big(\frac{p+4}{p+2}\Big)^{\frac{p}{p+4}}
\ \mbox{and}\
\lim\limits_{k\to\infty}
\frac{|z_{b_k}-x_0|}{\epsilon_{b_k}|\ln\epsilon_{b_k}|}
=\frac{p+2}{2}.
$$
\end{thm}

\begin{rem}
Theorems \ref{thm:mass-concentration-minimizer-beta-inter} and \ref{thm:mass-concentration-minimizer-beta-boundary} yield that the mass of minimizers concentrates either at the flattest inner global minimum of $V(x)$ as $b\searrow0$ or near the boundary of $\Omega$ as $b\searrow0$, depending on $\mathcal Z_1\not=\emptyset$ or not.
Meanwhile, we point that the mass concentration near the boundary of $\Omega$ is faster than the mass concentration at an interior point of $\Omega$, which is related to the refined energy estimate \eqref{eqn:energy-estimates-ln} involving logarithmic-type perturbation.

%During the process of Theorem \ref{thm:mass-concentration-minimizer-beta-inter}-\ref{thm:mass-concentration-minimizer-beta-boundary}, by $e(b)\to e(0)$ as $b\searrow0$ given by Theorem \ref{thm:existence-nonexistence-e-0}, it is not difficult to prove that
%This point can be achieved by $e(b)\to\bar e(b)$ as $b\searrow0$ in the proof Theorem \ref{thm:mass-concentration-minimizer-beta-boundary}.
%We further establish the refine upper bound of $e(b)$ as $b\searrow0$, then the desire result can be obtained. Moreover, this result provides a detailed description of the limit behavior and explicit rate for the minimizers of $e(b)$ as $b\searrow0$.
\end{rem}

Finally, we study $\beta>\beta^*$ and also focus on $\mathcal Z_1\neq\emptyset$ or $\mathcal Z_1=\emptyset$. In particular, the similar mass concentration behaviors can be established.

\begin{thm}\label{thm:mass-concentration-minimizer-beta-strict}
Suppose that $V(x)$ satisfies \eqref{eqn:V-potential-x-x-i} and $\mathcal Z_1\neq\emptyset$. Let $u_b$ be a nonnegative minimizer of $e(b)$ for $b>0$ and $\beta>\beta^*$. For any sequence $\{b_k\}$ satisfying $b_k\searrow0$ as $k\rightarrow\infty$, there exists a subsequence, still denoted by $\{b_k\}$, such that each $u_{b_k}$ has a unique maximum point $z_{b_k}$ with $\lim_{k\rightarrow\infty}z_{b_k}=x_0\in\mathcal Z_1$
and
\begin{equation}\label{eqn:beta-strict-thm-inner-u-b-k-convergence}
u_{b_k}(\varepsilon_{b_k}x+z_{b_k})\rightarrow\frac{Q(|x|)}
{\sqrt{\beta^*}}
\ \mbox{in}\ H^1(\mathbb R^2)\ \mbox{as}\ k\rightarrow\infty,\ \mbox{where}\ \varepsilon_{b_k}=\Big(\frac{\beta^*{b_k}}{\beta-\beta^*}\Big)^{\frac{1}{2}}.
\end{equation}
Moreover,
$$
\lim\limits_{k\to\infty}\frac{e(b_k)-\bar e(b_k)}{\varepsilon_{b_k}^p}
=\frac{2\lambda^{p+2}}{p}
\ \mbox{and}\
\lim_{k\rightarrow\infty}\frac{|z_{b_k}-x_0|}
{\varepsilon_{b_k}}=0.
$$
\end{thm}

\begin{rem}
We point out that in order to obtain blow-up behavior \eqref{eqn:beta-strict-thm-inner-u-b-k-convergence}, the explicit form of $\bar e(b)$ in Proposition \ref{pro:bar-e-b-explicit-form-bar-u-b} and the prior estimate $e(b)-\bar e(b)\to0$ as $b\searrow0$ in Theorem \ref{thm:existence-nonexistence-e-0}-$(iii)$ are central, which are used to determine that the minimizers occur blow-up and ensure that usual blow-up analysis method is applicable. However, it is not sufficient to derive the detail concentration information. To this end, we give a more fine upper energy estimate of $e(b)$ with infinitesimal perturbation, see Lemma \ref{lem:e-b-upper-bound}, which allows us to obtain the concentration point $x_0\in\mathcal Z_1$.
\end{rem}

\begin{thm}\label{thm:mass-concentration-minimizer-beta-strict-boundary}
Suppose that $V(x)$ satisfies \eqref{eqn:V-potential-x-x-i} and $\mathcal Z_1=\emptyset$. Let $u_b$ be a nonnegative minimizer of $e(b)$ for $b>0$ and $\beta>\beta^*$. For any sequence $\{b_k\}$ satisfying $b_k\searrow0$ as $k\rightarrow\infty$, there exists a subsequence, still denoted by $\{b_k\}$, such that $u_{b_k}$ also satisfies \eqref{eqn:beta-strict-thm-inner-u-b-k-convergence} and the unique maximum point
$z_{b_k}\to x_0\in\mathcal Z_0$ as $k\to\infty$. Moreover,
$$
\lim\limits_{k\to\infty}\frac{e(b_k)-\bar e(b_k)}{\varepsilon_{b_k}^p|\ln\varepsilon_{b_k}|^p}
=\kappa\left(\frac{p+2}{2}\right)^p
\ \mbox{and}\
\lim\limits_{k\to\infty}\frac{|z_{b_k}-x_0|}{\varepsilon_{b_k}|\ln\varepsilon_{b_k}|}=\frac{p+2}{2}.
$$
\end{thm}

\begin{rem}
From Theorems \ref{thm:mass-concentration-minimizer-beta-strict} and \ref{thm:mass-concentration-minimizer-beta-strict-boundary}, we find an interesting fact that the blow-up rates $\varepsilon_{b_k}$ are uniform whether the mass of minimizers concentrates at an inner point or near the boundary of $\Omega$, which is different from the case of $\beta=\beta^*$. This is because that for $\mathcal Z_1\neq\emptyset$ or not, the energy estimates are only subtly different in infinitesimal perturbation parts. We believe that this phenomenon and our method can provide inspiration for other constraint minimization problems in bounded domains if the minimal energy diverges to $-\infty$.
\end{rem}

The remainder of this paper is organized as follows. In Section \ref{sec:existence-nonexistence}, we complete the proof of Theorem \ref{thm:existence-nonexistence-e-0}. In Section \ref{sec:mass-concentration}, we are concerned with the case $\beta=\beta^*$ and prove Theorems \ref{thm:mass-concentration-minimizer-beta-inter}-\ref{thm:mass-concentration-minimizer-beta-boundary} in Subsections \ref{subsection:beta-equiv-inner}-\ref{subsection:beta-equiv-boundary}, respectively. Section \ref{sec:mass-concentration-bar-e-beta-b} is devoted to investigating the case $\beta>\beta^*$ and establishing the proof of Theorems \ref{thm:mass-concentration-minimizer-beta-strict} and \ref{thm:mass-concentration-minimizer-beta-strict-boundary}.

\section{Preliminaries}\label{sec:existence-nonexistence}

This section is dedicated to accomplishing the proof of Theorem \ref{thm:existence-nonexistence-e-0}.
We recall that $Q$ is the unique positive radially symmetric solution of \eqref{eqn:Q-unique-radial} satisfying
\begin{equation}\label{eqn:nabla-Q-Q-2-norm-nonlinearities}
\int_{\mathbb{R}^{2}}|\nabla Q|^2\mathrm{d}x
=\int_{\mathbb{R}^{2}}|Q|^2\mathrm{d}x=\frac{1}{2}
\int_{\mathbb{R}^{2}}|Q|^4\mathrm{d}x,
\end{equation}
and it follows from \cite[Proposition 4.1] {Gidas-1981} that
\begin{equation}\label{eqn:Q-decay}
|\nabla Q(x)|,\ Q(x)=O(|x|^{-\frac{1}{2}}e^{-|x|})\ \mbox{as}\ |x|\rightarrow\infty.
\end{equation}

\noindent\textbf{Proof of Theorem \ref{thm:existence-nonexistence-e-0}.}
%It is easy to check that
%$e(b)$ has always a minimizer for any $\beta\geq0$ when $b>0$. The remaining proof of $(i)$-($ii$) can be seen in
%\cite[Theorem 1.2]{Luo-2019}.
It is sufficient to prove $(iii)$. First, we verify that $e(b)\to e(0)=0$ as $b\searrow0$ if $\beta=\beta^*$. Let $u_b$ be the nonnegative minimizer of $e(b)$ for $b>0$. By \eqref{eqn:GN-type-inequality},
\begin{equation}\label{eqn:e-b-geq-strict-0}
e(b)=E_b(u_b)\geq\int_\Omega V(x)u_b^2\mathrm{d}x+\frac{b}{2}\Big(\int_\Omega|\nabla u_b|^2\mathrm{d}x\Big)^2
>e(0).
\end{equation}
Note that $\Omega\subset \mathbb{R}^{2}$ is a bounded domain. We suppose that there exists at least one inner point $x_0$ such that $V(x_0)=0$, then there exists an open ball $B_{2R}(x_0)\subset\Omega$, where $R>0$ is small enough.
Letting
$\psi\in C_0^\infty(\mathbb{R}^{2})$ be a nonnegative cut-off function with $\psi(x)=1$ for $|x|\leq R$ and $\psi(x)=0$ for all $|x|\geq 2R$, we consider a trial function
$$
u_\tau(x)=\frac{A_\tau\tau}{\sqrt{\beta^*}}\psi(x-x_0)Q(\tau(x-x_0)),
$$
where $A_\tau$ is determined such that $\int_{\Omega}|u_\tau|^2\mathrm{d}x=1$.
From \eqref{eqn:Q-decay}, we see that
$$
\frac{1}{A_\tau^2}=\frac{1}{\beta^*}\int_\Omega\tau^2\psi^2(x-x_0)Q^2(\tau(x-x_0))\mathrm{d}x
=1-o(e^{-\tau R})\ \mbox{as}\ \tau\to\infty,
$$
which means
\begin{equation}\label{eqn:A-tau-e-estimate}
1\leq A_\tau^2\leq 1+o(e^{-\tau R})\ \mbox{as}\ \tau\to\infty.
\end{equation}
Using \eqref{eqn:nabla-Q-Q-2-norm-nonlinearities} and \eqref{eqn:Q-decay} again, we can also obtain that
\begin{equation}\label{eqn:delta-u-tau-estimate}
\int_\Omega|\nabla u_\tau|^2\mathrm{d}x
=\frac{A_\tau^2\tau^2}{\beta^*}\int_{\mathbb R^2}|\nabla Q|^2\mathrm{d}x+o(e^{-\tau R})=A_\tau^2\tau^2+o(e^{-\tau R})
\ \mbox{as}\ \tau\to\infty
\end{equation}
and
\begin{equation}\label{eqn:u-tau-4-estimate}
\int_\Omega|u_\tau|^4\mathrm{d}x=\frac{A_\tau^4\tau^2}{(\beta^*)^2}\int_{\mathbb R^2}|Q|^4\mathrm{d}x+o(e^{-2\tau R})=\frac{2A_\tau^4\tau^2}{\beta^*}+o(e^{-2\tau R})
\ \mbox{as}\ \tau\to\infty.
\end{equation}
By the dominated convergence theorem, we have
$$
\int_\Omega V(x)u_\tau^2\mathrm{d}x
=\frac{A_\tau^2\tau^2}{\beta^*}\int_{\Omega}V(x)\psi^2(x-x_0)Q^2(\tau(x-x_0))\mathrm{d}x
=V(x_0)+o(1)\ \mbox{as}\ \tau\to\infty.
$$
Then
\begin{equation}\label{eqn:in-order-to-cite-boundary}
0\leq e(b)\leq E_b(u_\tau)=A_\tau^2\tau^2+V(x_0)+\frac{b}{2}A_\tau^4\tau^4-A_\tau^4\tau^2+o(e^{-\frac{1}{2}\tau R})+o(1)
\ \mbox{as}\ \tau\to\infty.
\end{equation}
Choosing $\tau=b^{-\frac{1}{8}}\to\infty$ as $b\searrow0$, this yields that $e(b)\to e(0)=0$ as $b\searrow0$.

On the other hand, we know that $V(x)$ may vanish only on $\partial\Omega$, i.e., there is $x_0\in\partial\Omega$ with $V(x_0)=0$.
Since the domain $\Omega$ satisfies the interior ball condition, there exists an inner point $x_i\in\Omega$ and $R>0$ such that $B_{2R}(x_i)\subset\Omega$ and $x_0\in\partial B_{2R}(x_i)$. Take $R_\tau=\frac{C\ln\tau}{\tau}$ for $C>1$ and $x_\tau:=x_0-2R_\tau\vec{n}$, where $\tau>0$ and $\vec{n}$ is the outer unit normal vector of $x_0$. Then $x_0\in\partial B_{2R_\tau}(x_\tau)$, $B_{2R_\tau}(x_\tau)\subset B_{2R}(x_i)$, $R_\tau\rightarrow 0$ and $x_\tau\rightarrow x_0$ as $\tau\rightarrow\infty$. Let $\psi\in C_0^\infty(\mathbb{R}^{2})$ be a nonnegative cut-off function with
$\psi(x)=1$ for $|x|\leq 1$ and $\psi(x)=0$ for all $|x|\geq 2$. Choose
$$
u_\tau(x)=\frac{A_\tau\tau}{\sqrt{\beta^*}}\psi(\frac{x-x_\tau}{R_\tau})Q(\tau(x-x_\tau)),\ x\in\Omega,
$$
where $A_\tau$ is determined such that $\int_{\Omega}|u_\tau|^2\mathrm{d}x=1$. Similar to \eqref{eqn:A-tau-e-estimate}-\eqref{eqn:u-tau-4-estimate}, there hold
\begin{equation}\label{eqn:in-order-to-boundary-x-0-V-0-A-tau-estimate}
1\leq A_\tau^2\leq 1+o(\tau^{-2})\ \mbox{as}\ \tau\to\infty
\end{equation}
and
$$
\begin{aligned}
\int_\Omega|\nabla u_\tau|^2\mathrm{d}x-\frac{\beta}{2}\int_\Omega|u_\tau|^4\mathrm{d}x
&=\frac{A_\tau^2\tau^2}{\beta^*}\int_{\mathbb R^2}|\nabla Q|^2\mathrm{d}x
-\frac{A_\tau^4\tau^2}{2\beta^*}\int_{\Omega}Q^4\mathrm{d}x+o(1)\\
&=A_\tau^2\tau^2(1-A_\tau^2)
+o(1)\\
&=o(1)  \ \mbox{as}\ \tau\rightarrow\infty.
\end{aligned}
$$
Moreover,
$$
\int_\Omega V(x)u_\tau^2\mathrm{d}x
=V(x_0)+o(1)\ \mbox{as}\ \tau\to\infty.
$$
Thus, we also conclude that \eqref{eqn:in-order-to-cite-boundary} holds and the desired result is obtained.

Finally, we show $e(b)-\bar e(b)\to0$ as $b\searrow0$ for $\beta>\beta^*$.
By the definition of $e(b)$ and $\bar e(b)$ and extending $u_b=0$ in $\mathbb R^2\backslash\Omega$, we have
$$
\bar e(b)\leq\bar E_b(u_b)\leq E_b(u_b)=e(b).
$$
That is, $e(b)-\bar e(b)\geq0$.
Similarly, we assume that there exists at least one inner point $x_0\in\Omega$ satisfying $V(x_0)=0$.
Then there exists an open ball $B_{2R}(x_0)\subset\Omega$, where $R>0$ is small enough. Let $0\leq\psi(x)\in C_0^\infty(\mathbb R^2)$ be a cut-off function such that
$\psi(x)=1$ for $|x|\leq R$, $\psi(x)=0$ for $|x|\geq2R$ and $0\leq\psi(x)\leq1$ for $R\leq|x|\leq2R$.
Choose the trial function
\begin{equation}\label{eqn:e-b-bar-e-b-0}
\phi_b(x)=\frac{A_br_b^{\frac{1}{2}}}{\sqrt{\beta^*}}\psi(x-x_0)Q(r_b^{\frac{1}{2}}(x-x_0)),
\end{equation}
where $r_b$ is defined in Proposition \ref{pro:bar-e-b-explicit-form-bar-u-b} and $A_b>0$ is determined such that $\int_{\Omega}|\phi_b|^2\mathrm{d}x=1$. By the exponential decay \eqref{eqn:Q-decay} of $Q$, we can deduce that
$$
\begin{aligned}
\frac{1}{A_b^2}&=\frac{r_b}{\beta^*}\int_{\Omega}\psi^2(x-x_0)Q^2(r_b^{\frac{1}{2}}(x-x_0))\mathrm{d}x\\
&=\frac{\int_{\mathbb R^2}Q^2(x)\mathrm{d}x}{\beta^*}
+\frac{1}{\beta^*}\int_{\mathbb R^2}(\psi^2(r_b^{-\frac{1}{2}}x)-1)Q^2(x)\mathrm{d}x\\
&\geq1-o(e^{-r_b^{\frac{1}{2}}R})\ \mbox{as}\ b\searrow0,
\end{aligned}
$$
which signifies that
\begin{equation}\label{eqn:A-b-1}
1\leq A_b^2\leq1+o(e^{-r_b^{\frac{1}{2}}R})\ \mbox{as}\ b\searrow0.
\end{equation}
Then it follows from \eqref{eqn:bar-u-b-form-r-b-form} that
\begin{equation}\label{eqn:phi-b-bar-u-b}
\begin{aligned}
\int_\Omega|\phi_b(x)|^4\mathrm{d}x
&=\int_{\mathbb R^2}\frac{A_b^4r_b}{(\beta^*)^2}\psi^4(r_b^{-\frac{1}{2}}x)Q^4(x)\mathrm{d}x\\
&\geq\int_{|x|\leq r_b^{-\frac{1}{2}}R}\frac{A_b^4r_b}{(\beta^*)^2}Q^4(x)\mathrm{d}x\\
&=\int_{\mathbb R^2}\frac{A_b^4r_b}{(\beta^*)^2}Q^4(x)\mathrm{d}x
-\int_{|x|\geq r_b^{\frac{1}{2}}R}\frac{A_b^4r_b}{(\beta^*)^2}Q^4(x)\mathrm{d}x\\
&\geq\int_{\mathbb R^2}\frac{r_b}{(\beta^*)^2}Q^4(x)\mathrm{d}x-o(e^{-r_b^{\frac{1}{2}}R})\\
&=\int_{\mathbb R^2}|\bar u_b|^4\mathrm{d}x-o(e^{-r_b^{\frac{1}{2}}R})\ \mbox{as}\ b\searrow0.
\end{aligned}
\end{equation}
Similarly, we also have
\begin{equation}\label{eqn:nabla-phi-b-nabla-bar-u-b}
\int_{\Omega}|\nabla\phi_b|^2\mathrm{d}x
\leq\frac{r_b}{\beta^*}\int_{\mathbb R^2}|\nabla Q|^2\mathrm{d}x
+o(e^{-\frac{1}{2}r_b^{\frac{1}{2}}R})
=\int_{\mathbb R^2}|\nabla\bar u_b|^2\mathrm{d}x+o(e^{-\frac{1}{2}r_b^{\frac{1}{2}}R})\ \mbox{as}\ b\searrow0.
\end{equation}
Moreover, we can check that
\begin{equation}\label{eqn:in-order-to-boundary-V-x-0-0}
\begin{aligned}
\int_{\Omega}V(x)\phi_b^2(x)\mathrm{d}x
&=\frac{A_b^2r_b}{\beta^*}\int_{\mathbb R^2}V(x)\psi^2(x-x_0)Q^2(r_b^{\frac{1}{2}}(x-x_0))\mathrm{d}x\\
&=\frac{A_b^2}{\beta^*}\int_{\mathbb R^2}V(r_b^{-\frac{1}{2}}x+x_0)\psi^2(r_b^{-\frac{1}{2}}x)Q^2(x)\mathrm{d}x\\
&=V(x_0)+o(1)\ \mbox{as}\ b\searrow0.
\end{aligned}
\end{equation}
The above estimates yield that
\begin{equation}\label{eqn:in-order-to-e-b-bar-e-b-0}
\begin{aligned}
e(b)-\bar e(b)
\leq\bar E_b(\phi_b)\!-\!\bar E(\bar u_b)\!+\!\int_\Omega V(x)\phi_b^2\mathrm{d}x
\leq V(x_0)+o(e^{-\frac{1}{4}r_b^{\frac{1}{2}}R})+o(1)\rightarrow0\ \mbox{as}\ b\searrow0.
\end{aligned}
\end{equation}

Now, we suppose that $V(x)$ vanishes only on $\partial\Omega$, i.e., there exists $x_0\in\partial\Omega$ with $V(x_0)=0$. Using the same definitions as \eqref{eqn:in-order-to-cite-boundary} below,
we shall have $x_0\in\partial B_{2R_\tau}(x_\tau)$, $B_{2R_\tau}(x_\tau)\subset B_{2R}(x_i)$, $R_\tau\rightarrow 0$ and $x_\tau\rightarrow x_0$ as $\tau\rightarrow\infty$. Let $\psi\in C_0^\infty(\mathbb{R}^{2})$ be a nonnegative cut-off function with $\psi(x)=1$ for $|x|\leq 1$ and $\psi(x)=0$ for all $|x|\geq 2$.
Choose $\tau=r_b^{\frac{1}{2}}\to\infty$ as $b\searrow0$ and
$$
\phi_\tau(x)=\frac{A_\tau\tau}{\sqrt{\beta^*}}\psi(\frac{x-x_\tau}{R_\tau})Q(\tau(x-x_\tau)),
$$
where $A_\tau>0$ is determined such that $\int_{\Omega}|\phi_\tau|^2\mathrm{d}x=1$.
In view of \eqref{eqn:in-order-to-boundary-x-0-V-0-A-tau-estimate}, we also have
$$
1\leq A_\tau^2\leq1+o(\tau^{-2})\ \mbox{as}\ b\searrow0.
$$
Similar to \eqref{eqn:phi-b-bar-u-b}-\eqref{eqn:in-order-to-boundary-V-x-0-0}, we can obtain that
$$
\int_\Omega|\phi_\tau(x)|^4\mathrm{d}x
\geq\int_{\mathbb R^2}|\bar u_b|^4\mathrm{d}x-o(1),\
\int_{\Omega}|\nabla\phi_\tau|^2\mathrm{d}x
\leq\int_{\mathbb R^2}|\nabla\bar u_b|^2\mathrm{d}x+o(1)\ \mbox{as}\ b\searrow0,
$$
and
$$
\begin{aligned}
\int_{\Omega}V(x)\phi_\tau^2(x)\mathrm{d}x
=V(x_0)+o(1)\ \mbox{as}\ b\searrow0.
\end{aligned}
$$
The estimates show that \eqref{eqn:in-order-to-e-b-bar-e-b-0} still holds and the proof is complete.

\section{Mass concentration of minimizers for $\beta=\beta^*$}\label{sec:mass-concentration}

In this section, we consider the case $\beta=\beta^*$ and prove Theorems \ref{thm:mass-concentration-minimizer-beta-inter}-\ref{thm:mass-concentration-minimizer-beta-boundary}.
We first give the blow-up behavior of minimizers as $b\searrow0$.

\begin{lem}\label{lem:blow-up-beta-equiv}
Suppose that $V(x)$ satisfies $(V_1)$ and let $u_b$ be a nonnegative minimizer of $e(b)$.
\begin{enumerate}
 \item [($i$)] There hold that
$$
\int_\Omega|\nabla u_b|^2\mathrm{d}x\to\infty\ \mbox{as}\ b\searrow0,
$$
$$
b\Big(\int_\Omega|\nabla u_b|^2\mathrm{d}x\Big)^2\rightarrow0\ \mbox{and}\ \int_\Omega V(x)u_b^2\mathrm{d}x\to 0
\ \mbox{as}\ b\searrow0.
$$
 \item [($ii$)] Define
$$
\epsilon_b^{-2}:=\int_\Omega|\nabla u_b|^2\mathrm{d}x,\ \mbox{where}\ \epsilon_b>0.
$$
Let $z_b$ be a maximum point of $u_b$ in $\Omega$ and define
$$
\Omega_b=\{x\in\mathbb R^2:(\epsilon_bx+z_b)\in\Omega\},\ \varphi_b(x):=\epsilon_bu_b(\epsilon_bx+z_b),\ x\in\Omega_b,
$$
then there exist positive constants $R>0$ and $\alpha_0>0$ such that
\begin{equation}\label{eqn:omega-b-2-beta-strict-0-equiv}
\liminf\limits_{b\searrow0}\int_{B_{2R}(0)\cap\Omega_b}|\varphi_b|^2\mathrm{d}x\geq\alpha_0>0.
\end{equation}
\end{enumerate}
\end{lem}

\begin{proof}
$(i)$ Combining Theorem \ref{thm:existence-nonexistence-e-0}-$(iii)$ and \eqref{eqn:e-b-geq-strict-0}, we can directly know that
\begin{equation}\label{eqn:in-order-to-subsequent-blow-up}
b\Big(\int_{\Omega}|\nabla u_b|^2\mathrm{d}x\Big)^2\rightarrow0\ \mbox{and}\ \int_\Omega V(x)u_b^2\mathrm{d}x\to 0\ \mbox{as}\ b\searrow0.
\end{equation}
Now, we prove $\int_\Omega|\nabla u_b|^2\mathrm{d}x\to\infty\ \mbox{as}\ b\searrow0$. By contradiction, there exists a sequence of $\{b_k\}$ such that $\{u_{b_k}\}$ is bounded in $H_0^1(\Omega)$. According to the compact embedding $H_0^1(\Omega)\hookrightarrow L^r(\Omega)$ for $r\geq 1$, up to a subsequence, we have
$$
u_{b_k}\rightharpoonup u\ \mbox{in}\ H_0^1(\Omega)\ \mbox{and}\ u_{b_k}\to u\ \mbox{in}\ L^r(\Omega)\
\mbox{for}\ r\in[1,\infty).
$$
Therefore, applying Fatou's lemma, we get $u$ is a minimizer of $e(0)$, which contradicts with Theorem \ref{thm:existence-nonexistence-e-0}-($i$).

($ii$) In view of the definition of $\epsilon_b$ and $(i)$, it is obvious that
\begin{equation}\label{eqn:epsilon-b-0}
\epsilon_b\to0\ \mbox{as}\ b\searrow0.
\end{equation}
We can also know from the definition of $\varphi_b$ that
\begin{equation}\label{eqn:nabla-varphi-1}
\int_{\Omega_b}|\nabla \varphi_b|^2\mathrm{d}x=\epsilon_b^2\int_\Omega|\nabla u_b|^2\mathrm{d}x=1.
\end{equation}
Since $u_b$ is a nonnegative minimizer for $e(b)$, $u_b$ satisfies the Euler-Lagrange equation
\begin{equation}\label{eqn:u-b-equation}
\left\{\begin{array}{ll}
-(1+b\int_{\Omega}|\nabla u_b|^2\mathrm{d}x)\Delta u_b+V(x)u_b=\mu_b u_b+\beta u_b^3  &\mbox{in}\ {\Omega}, \\[0.1cm]
 u_b=0&\mbox{on}\ {\partial\Omega}, \\[0.1cm]
\end{array}
\right.
\end{equation}
where $\mu_b\in\mathbb R$ is the associated Lagrange multiplier and
$$
\mu_b=\int_{\Omega}|\nabla u_b|^2\mathrm{d}x+\int_\Omega V(x)u_b^2\mathrm{d}x
+b\Big(\int_\Omega|\nabla u_b|^2\mathrm{d}x\Big)^2-\beta\int_\Omega u_b^4\mathrm{d}x.
$$
Note that
$$
\epsilon_b^2e(b)=\epsilon_b^2\int_{\Omega}|\nabla u_b|^2\mathrm{d}x
+\frac{b}{2}\epsilon_b^2\left(\int_\Omega|\nabla u_b|^2\mathrm{d}x\right)^2
+\epsilon_b^2\int_\Omega V(x)u_b^2\mathrm{d}x-\frac{\beta}{2}\epsilon_b^2\int_{\Omega}|u_b|^4\mathrm{d}x.
$$
Together with $(i)$ and \eqref{eqn:epsilon-b-0}, we derive that
$$
\int_{\Omega_b}|\varphi_b|^4\mathrm{d}x=\epsilon_b^2\int_\Omega|u_b|^4\mathrm{d}x\to\frac{2}{\beta}
\ \mbox{as}\ b\searrow0.
$$
Further,
\begin{equation}\label{eqn:mu-b-estimate-beta-equiv}
\mu_b\epsilon_b^2=\epsilon_b^2e(b)-\frac{b}{2}\epsilon_b^2\Big(\int_\Omega|\nabla u_b|^2\mathrm{d}x\Big)^2
-\frac{\beta}{2}\epsilon_b^2\int_\Omega|u_b|^4\mathrm{d}x
\rightarrow-1\ \mbox{as}\ b\searrow0.
\end{equation}
Following \eqref{eqn:u-b-equation}, $\varphi_b$ satisfies the following equation
\begin{equation}\label{eqn:varphi-b-equation}
\left\{\begin{array}{ll}
-(1+b\epsilon_b^{-2}\int_{\Omega_b}|\nabla\varphi_b|^2\mathrm{d}x)\Delta \varphi_b+\epsilon_b^2V(\epsilon_bx+z_b)\varphi_b=\epsilon_b^2\mu_b\varphi_b+\beta\varphi_b^3  &\mbox{in}\ {\Omega_b}, \\[0.1cm]
\varphi_b=0&\mbox{on}\ {\partial\Omega_b}. \\[0.1cm]
\end{array}
\right.
\end{equation}
In light of $V(x)\geq0$ and \eqref{eqn:mu-b-estimate-beta-equiv}, we can infer that
\begin{equation}\label{eqn:varphi-b-subsolution-equation}
-\Delta\varphi_b-c(x)\varphi_b\leq0\ \mbox{in}\ \Omega_b,
\end{equation}
where $c(x)=\beta\varphi_b^2(x)$.
Since $z_b$ is a maximum point of $u_b$, then 0 is a maximum point of $\varphi_b$. Thus $-\Delta\varphi_b(0)\geq0$. According to \eqref{eqn:mu-b-estimate-beta-equiv}-\eqref{eqn:varphi-b-equation} and $V(x)\geq0$, we have
$$
\beta\varphi_b^3(0)\geq-\mu_b\epsilon_b^2\varphi_b(0)\geq\frac{\beta}{2\beta^*}\varphi_b(0),
$$
which gives that
\begin{equation}\label{eqn:varphi-b-0-strict-0}
\varphi_b(0)\geq\sqrt{\frac{1}{2\beta^*}}\ \mbox{as}\ b\searrow0.
\end{equation}
Next, we consider two cases as follows.

{\bf Case 1.} There exists a constant $R>0$ such that $B_{2R}(0)\subset\Omega_b$ as $b\searrow0$. Since $\varphi_b\in H^1(\Omega_b)$ is a subsolution of \eqref{eqn:varphi-b-subsolution-equation}, then using the De Giorgi-Nash-Morser theory {{\cite[Theorem 4.1] {Han-2011}}}, we see
\begin{equation*}\label{eqn:w-max-B-0-upper-2-norm}
\max\limits_{B_R(0)}\varphi_b\leq C\left(\int_{B_{2R}(0)}|\varphi_b|^2\mathrm{d}x\right)^\frac{1}{2}
=C\left(\int_{B_{2R}(0)\cap\Omega_b}|\varphi_b|^2\mathrm{d}x\right)^\frac{1}{2},
\end{equation*}
where $C>0$ depends only on the bound of $\|\varphi_b\|_{L^4(\Omega_b)}$.
Hence, by \eqref{eqn:varphi-b-0-strict-0}, there exists a positive constant $\alpha_0$ such that
$$
\liminf\limits_{b\searrow0}\int_{B_{2R}\cap\Omega_b}|\varphi_b|^2\mathrm{d}x\geq\alpha_0>0.
$$

{\bf Case 2.} For any constant $R>0$, $B_{2R}(0)\not\subset\Omega_b$ as $b\searrow0$. In this case, let
$\hat\varphi_b(x)\equiv\varphi_b(x)$ for $x\in\Omega_b$ and $\hat\varphi_b(x)\equiv0$ for
$x\in\mathbb R^2\backslash\Omega_b$. Then $\hat\varphi_b(x)$ satisfies
$$
-\Delta\hat\varphi_b-\hat c(x)\hat\varphi_b\leq0\ \mbox{in}\ \mathbb R^2,
$$
where $\hat{c}(x)\equiv\beta\varphi_b^2$ in $\Omega_b$ and $\hat{c}(x)\equiv0$ in
$\mathbb R^2\backslash\Omega_b$. In particular, for any constant $R>0$, we obtain $\sup_{x\in B_R(0)}\hat\varphi_b(x)\geq \varphi_b(0)\geq\sqrt{\frac{1}{2\beta^*}}$. Similar to {\bf Case 1}, we also have
$$
\sup\limits_{B_R(0)}\hat\varphi_b\leq C\left(\int_{B_{2R}(0)}
|\hat\varphi_b|^2\mathrm{d}x\right)^\frac{1}{2}=C\left(\int_{B_{2R}(0)\cap\Omega_b}|
\varphi_b|^2\mathrm{d}x\right)^\frac{1}{2}.
$$
This completes the proof.
\end{proof}

\begin{lem}\label{lem:varphi-b-k-Q-convergence}
For any sequence $\{b_k\}$ satisfying $b_k\searrow0$ as $k\rightarrow\infty$, there exists a subsequence, still denoted by $\{b_k\}$, such that
\begin{itemize}
  \item [$(i)$] $\lim_{k\rightarrow\infty}z_{b_k}=x_{i_0}\in\bar{\Omega}$ satisfying $V(x_{i_0})=0$;
  \item [$(ii)$] If
$$
\Omega_0=\lim\limits_{k\rightarrow\infty}\Omega_{b_k}=\mathbb R^2,
$$
then $u_{b_k}$ has a unique maximum point $z_{b_k}$ and
\begin{equation*}\label{eqn:w-a-Q}
\lim\limits_{k\rightarrow\infty}\varphi_{b_k}(x)=\frac{Q(|x|)}{\sqrt{\beta^*}}\ \mbox{in}\ H^1(\mathbb R^2).
\end{equation*}
\end{itemize}
\end{lem}

\begin{proof}
($i$) Since $\{z_b\}\subset\Omega$ and $\Omega$ is a bounded domain, then the sequence $\{z_b\}$ is bounded uniformly as $b\searrow0$. For any sequence $\{z_{b_k}\}$ satisfying $b_k\searrow0$, there is a subsequence, still denoted by $\{z_{b_k}\}$, so that $z_{b_k}\rightarrow x_{i_0}$ as $b_k\searrow0$ with $x_{i_0}\in\bar\Omega$. Now, we assume that $V(x_{i_0})\geq\eta>0$, then we can infer from \eqref{eqn:omega-b-2-beta-strict-0-equiv} and Fatou's lemma that
$$
\begin{aligned}
&\quad\liminf\limits_{b_k\searrow0}\int_{\Omega_{b_k}}
V(\epsilon_{b_k}x+z_{b_k})|\varphi_{b_k}|^2\mathrm{d}x
\geq\int_{B_{2R}(0)\cap\Omega_{b_k}}\liminf\limits_{b_k\searrow0}V(\epsilon_{b_k}x+z_{b_k})|\varphi_{b_k}|^2
\mathrm{d}x\geq\eta\alpha_0>0,
\end{aligned}
$$
which is a contradiction with Lemma \ref{lem:blow-up-beta-equiv} since it shows that
$$
\int_{\Omega_{b_k}}V(\epsilon_{b_k}x+z_{b_k})|\varphi_{b_k}|^2\mathrm{d}x
=\int_\Omega V(x)u_{b_k}^2\mathrm{d}x\rightarrow0\ \mbox{as}\ b_k\searrow0.
$$
Hence $(i)$ holds.

($ii$) From the definition of $\varphi_{b_k}$ and \eqref{eqn:nabla-varphi-1}, we have that $\{\varphi_{b_k}\}$ is bounded in $H^1(\Omega_{b_k})$. Under the assumption $\Omega_0=\lim_{k\rightarrow\infty}\Omega_{b_k}=\mathbb R^2$, taking a convergent subsequence $\{b_k\}$ if necessary, there exists a nonnegative function $\varphi_0\in H^1(\mathbb R^2)$ such that $\varphi_{b_k}\rightharpoonup\varphi_0$ in $H^1(\mathbb R^2)$ as $k\to\infty$. Since $\varphi_{b_k}$ satisfies \eqref{eqn:varphi-b-equation}, by passing to the weak limit of $\varphi_{b_k}$, we can check from \eqref{eqn:mu-b-estimate-beta-equiv} that $\varphi_0$ satisfies the following equation
\begin{equation}\label{eqn:varphi-0-equation}
-\Delta\varphi_0+\varphi_0=\beta^*\varphi_0^3\ \mbox{in}\ \Omega_0.
\end{equation}
In view of $\Omega_0=\mathbb R^2$, from \eqref{eqn:omega-b-2-beta-strict-0-equiv} we know that $\varphi_0\neq0$ and thus $\varphi_0>0$ by the strong maximum principle. Applying a simple rescaling,
$$
\varphi_0(x)=\frac{Q(x-y_0)}{\sqrt{\beta^*}}\ \mbox{for\ some}\ y_0\in\mathbb R^2,
$$
which yields that $\int_{\mathbb R^2}|\varphi_0|^2\mathrm{d}x=1$. From now on, when necessary we shall extend $\varphi_{b_k}$ to $\mathbb R^2$ by letting $\varphi_{b_k}\equiv0$ on $\mathbb R^2\backslash\Omega$. Then we can derive that
$$
\varphi_{b_k}\to\varphi_0\ \mbox{in}\ L^2(\mathbb R^2)\ \mbox{as}\ k\rightarrow\infty.
$$
By interpolation inequality and the boundedness of $\varphi_{b_k}$ in $H^1(\mathbb R^2)$, we directly obtain
$$
\varphi_{b_k}\to\varphi_0\ \mbox{in}\ L^r(\mathbb R^2)\ \mbox{with}\ r\in[2,\infty)\ \mbox{as}\ k\rightarrow\infty.
$$
Following \eqref{eqn:varphi-b-equation} and \eqref{eqn:varphi-0-equation}, it results that
$$
\varphi_{b_k}\to\varphi_0\ \mbox{in}\ H^1(\mathbb R^2)\ \mbox{as}\ k\rightarrow\infty.
$$
Note that $V(x)\in C^\alpha(\Omega)$ and $\Omega$ is a bounded domain in $\mathbb R^2$ with $C^1$ boundary. It follows from \eqref{eqn:varphi-b-equation} that $\varphi_{b_k}\in C_{loc}^{2,\alpha_1}(\Omega_{b_k})$ for some $\alpha_1\in(0,1)$, which implies that
\begin{equation}\label{eqn:varphi-b-k-varphi-0-C2}
\varphi_{b_k}\rightarrow\varphi_0\ \mbox{in}\ C_{loc}^2(\mathbb R^2)\ \mbox{as}\ k\rightarrow\infty.
\end{equation}
Because the origin is a critical point of $\varphi_{b_k}$ for all $k>0$, then it is also a critical point of $\varphi_0$. Hence, $\varphi_0$ is spherically symmetric about the origin and
\begin{equation*}\label{eqn:w-0-Q-2}
\varphi_0(x)=\frac{Q(|x|)}{\sqrt{\beta^*}}.
\end{equation*}
Moreover, it follows from \eqref{eqn:varphi-b-k-varphi-0-C2} that all global maximum point of $\varphi_{a_k}$ must stay in a ball $B_\delta(0)$ for some small constant $\delta>0$. In view of $\varphi_0''(0)<0$, we can see that $\varphi_{b_k}''(|x|)<0$ in $B_\delta(0)$ for $k$ sufficiently large. Then by \cite[Lemma 4.2]{Ni-1991}, we know that $\varphi_{b_k}$ has no critical point other than the origin in $B_{\delta}(0)$ as $k>0$ large enough, which yields that the uniqueness of global maximum points for $\varphi_{b_k}$. Hence, $z_{b_k}$ is the unique maximum point of $u_{b_k}$ and the proof is complete.
\end{proof}

\subsection{Mass concentration at an interior point}\label{subsection:beta-equiv-inner}

In this subsection, under the assumptions \eqref{eqn:V-potential-x-x-i} and $\mathcal Z_1\neq\emptyset$, Theorem \ref{thm:mass-concentration-minimizer-beta-inter} is obtained. First, we give the following upper energy estimate of $e(b)$ as $b\searrow0$.

\begin{lem}\label{lem:beta-equiv-inner-upper-bound}
Suppose that $V(x)$ satisfies \eqref{eqn:V-potential-x-x-i} and $\mathcal Z_1\neq\emptyset$, then
$$
e(b)\leq\frac{p+2}{2p}\lambda^{\frac{4(p+2)}{p+4}}b^{\frac{p}{p+4}}\ \mbox{as}\ b\searrow0.
$$
\end{lem}

\begin{proof}
Since $\mathcal Z_1\neq\emptyset$ and $\Omega$ is a bounded domain, then there exists an open ball $B_{2R}(x_i)\subset\Omega$, where $x_i\in\mathcal Z_1$ for $i\in\{1,2,\cdots,n\}$ and $R>0$ is sufficiently small.
Because $V(x)$ satisfies \eqref{eqn:V-potential-x-x-i}, then there also exists $M>0$ such that
$$
V(x)\leq M|x-x_i|^p\ \mbox{in}\ \Omega.
$$
This infers that
\begin{equation}\label{eqn:V-leq-M-p-x-x-i-inner-beta-equiv}
V(\frac{x}{\tau}+x_i)\leq M\tau^{-p}|x|^p\
\mbox{for\ any}\ x\in\Omega_\tau:=\{x\in\mathbb R^2:(\frac{x}{\tau}+x_i)\in\Omega\}.
\end{equation}
Take $\psi\in C_0^\infty(\mathbb{R}^{2})$ be a nonnegative cut-off function with $\psi(x)=1$ for $|x|\leq R$ and $\psi(x)=0$ for all $|x|\geq 2R$ and choose a test function
$$
\Phi_\tau(x)=\frac{A_\tau\tau}{\sqrt{\beta^*}}\psi(x-x_i)Q(\tau(x-x_i)),
$$
where $A_\tau$ is determined such that $\int_{\Omega}|\Phi_\tau|^2\mathrm{d}x=1$.
Similar to \eqref{eqn:A-tau-e-estimate}-\eqref{eqn:u-tau-4-estimate}, we have that
$$
1\leq A_\tau^2\leq 1+o(e^{-\tau R})\ \mbox{as}\ \tau\to\infty,
$$
$$
\int_\Omega|\nabla\Phi_\tau|^2\mathrm{d}x
=A_\tau^2\tau^2+o(e^{-\tau R})\ \mbox{and}\
\int_\Omega|\Phi_\tau|^4\mathrm{d}x=\frac{2A_\tau^4\tau^2}{\beta^*}+o(e^{-2\tau R})\ \mbox{as}\ \tau\to\infty.
$$
Under the assumption \eqref{eqn:V-potential-x-x-i}, we can deduce from \eqref{eqn:V-leq-M-p-x-x-i-inner-beta-equiv} that
$$
\begin{aligned}
\int_{\Omega}V(x)\Phi_\tau^2\mathrm{d}x
%&=\frac{A^2_\tau\tau^3}{|Q|_2^2}\int_{\Omega}V(x)\varphi^2(x-x_i)Q^2(\tau(x-x_i))\mathrm{d}x\\
&=\frac{A^2_\tau}{\beta^*}\int_{\Omega_\tau}V(\frac{x}{\tau}+x_i)\psi^2(\frac{x}{\tau})Q^2(x)\mathrm{d}x\\
&=\frac{A^2_\tau}{\beta^*}\int_{|x|\leq \tau R}V(\frac{x}{\tau}+x_i)Q^2(x)\mathrm{d}x
+\frac{A^2_\tau}{\beta^*}\int_{\tau R\leq|x|\leq 2\tau R}V(\frac{x}{\tau}+x_i)\psi^2(\frac{x}{\tau})Q^2(x)\mathrm{d}x\\
&\leq\frac{A^2_\tau}{\beta^*}\int_{\mathbb R^2}V(\frac{x}{\tau}+x_i)Q^2(x)\mathrm{d}x
-\frac{A^2_\tau}{\beta^*}\int_{|x|\geq\tau R}V(\frac{x}{\tau}+x_i)Q^2(x)\mathrm{d}x\\
&\quad\ +\frac{A^2_\tau}{\beta^*}M\tau^{-p}\int_{\tau R\leq|x|\leq 2\tau R}|x|^pQ^2(x)\mathrm{d}x\\
&=\frac{\kappa_i\tau^{-p}}{\beta^*}\int_{\mathbb R^2}|x|^pQ^2(x)\mathrm{d}x+o(\tau^{-p})\ \mbox{as}\ \tau\rightarrow\infty.
\end{aligned}
$$
The above estimates yield that
$$
e(b)\leq E_b(\Phi_\tau)\leq\frac{b}{2}\tau^4
+\frac{\kappa_i\tau^{-p}}{\beta^*}\int_{\mathbb R^2}|x|^pQ^2(x)\mathrm{d}x+o(\tau^{-p})\ \mbox{as}\ \tau\rightarrow\infty.
$$
By choosing
$$
\tau=\Big(\frac{p\kappa_i}{2b\beta^*}\int_{\mathbb R^2}|x|^pQ^2(x)\mathrm{d}x\Big)^{\frac{1}{p+4}}
=\lambda_i^{\frac{p+2}{p+4}}b^{-\frac{1}{p+4}},\ i\in\{1,2,\cdots,n\},
$$
we can conclude that
$$
e(b)\leq\frac{p+2}{2p}\lambda^{\frac{4(p+2)}{p+4}}b^{\frac{p}{p+4}}\ \mbox{as}\ b\searrow0.
$$
\end{proof}

Based on Lemma \ref{lem:varphi-b-k-Q-convergence}, we know that $\lim_{k\rightarrow\infty}z_{b_k}=x_{i}\in\bar{\Omega}$ with $V(x_{i})=0$. To further prove $z_{b_k}\rightarrow x_i$ as $k\to\infty$ for some $x_i\in\mathcal Z_1$,
we need to exclude the case of $x_i\in\partial\Omega$, which requires more specific information. In this regard, we present the following exponential estimate given by \cite{Li-2021}, and this estimate is also important for $\beta>\beta^*$. Meanwhile, we mention that the notation $C$ denotes a positive constant, which might be changed from line to line and even in the same line.

\begin{lem}{\rm(\!\!\cite[Proposition 3.4]{Li-2021})}\label{lem:similar-to-Guo-Prop-3.1-beta-equiv}
Under the settings of Lemma \ref{lem:blow-up-beta-equiv}, suppose that up to a subsequence, $z_{b_k}\rightarrow x_i$ as $k\rightarrow\infty$ for some $x_i\in\partial\Omega$ with $V(x_i)=0$. Then there exists a sequence $\{\rho_k\}$ satisfying $\rho_k\rightarrow0$ as $k\rightarrow\infty$ such that
$$
\frac{\int_{\Omega_{b_k}}|\varphi_{b_k}|^{4}\mathrm{d}x}
{\int_{\Omega_{b_k}}|\nabla\varphi_{b_k}|^2\mathrm{d}x}
\leq\frac{2}{\beta^*}-Ce^{-\frac{2}{1+\rho_k}\frac{|z_{b_k}-x_i|}{\epsilon_{b_k}}}
\ \mbox{as}\ k\rightarrow\infty.
$$
\end{lem}

\begin{lem}\label{lem:z-b-k-to-x-i-0-mathcal-Z-1}
Suppose $V(x)$ satisfies \eqref{eqn:V-potential-x-x-i} and $\mathcal Z_1\neq\emptyset$, then for the subsequence $\{b_k\}$ obtained in Lemma \ref{lem:varphi-b-k-Q-convergence}, there holds $z_{b_k}\rightarrow x_i$ for some $x_i\in\mathcal Z_1$.
\end{lem}

\begin{proof}
From Lemma \ref{lem:varphi-b-k-Q-convergence}, we have $z_{b_k}\to x_{i_0}$ as $k\to\infty$ with $V(x_{i_0})=0$. We first show that
$\{\frac{|z_{b_k}-x_{i_0}|}{\epsilon_{b_k}}\}$ is bounded uniformly as $k\to\infty$. By contradiction, there exists a subsequence, still denoted by $\{b_k\}$, of $\{b_k\}$ such that $\frac{|z_{b_k}-x_{i_0}|}{\epsilon_{b_k}}\to\infty$ as $k\to\infty$. Then for any large constant $M$, by using \eqref{eqn:omega-b-2-beta-strict-0-equiv} and Fatou's lemma, we can infer
$$
\begin{aligned}
&\quad\ \liminf\limits_{\epsilon_{b_k}\rightarrow0}\frac{1}{\epsilon_{b_k}^{p_{i_0}}}
\int_{\mathbb R^2}V(\epsilon_{b_k}x+z_{b_k})|\varphi_{b_k}|^2\mathrm{d}x\\
&\geq C\int_{B_{2R}(0)\cap\Omega_{b_k}}\liminf\limits_{\epsilon_{b_k}\rightarrow0}\Big|x
+\frac{z_{b_k}-x_{i_0}}{\epsilon_{b_k}}\Big|^{p_{i_0}}
\prod\limits_{j=1,j\neq i_0}^n|\epsilon_{b_k}x+z_{b_k}-x_j|^{p_i}|\varphi_{b_k}|^2
\mathrm{d}x\geq M.
\end{aligned}
$$
Hence it follows from
\eqref{eqn:GN-type-inequality} that
$$
\begin{aligned}
e(b_k)&\geq\frac{b_k}{2}\Big(\int_\Omega|\nabla u_{b_k}|^2\mathrm{d}x\Big)^2+\int_\Omega V(x)u_{b_k}^2\mathrm{d}x\\
&\geq\frac{b_k}{2\epsilon_{b_k}^4}+M\epsilon_{b_k}^{p_{i_0}}
\geq\Big[\frac{1}{2}\Big(\frac{p_{i_0}}{2}\Big)^{\frac{4}{p_{i_0}+4}}
+\Big(\frac{2}{p_{i_0}}\Big)^{\frac{p_{i_0}}{p_{i_0}+4}}\Big]
M^{\frac{4}{p_{i_0}+4}}b_k^{\frac{p_{i_0}}{p_{i_0}+4}}\ \mbox{as}\ b_k\to0,
\end{aligned}
$$
which contradicts with Lemma \ref{lem:beta-equiv-inner-upper-bound} since $M$ is arbitrarily large.

Now we exclude the case that $z_{b_k}\to x_{i_0}$ as $k\to\infty$ with $x_{i_0}\in\partial\Omega$. If this case occurs, from
Lemma \ref{lem:similar-to-Guo-Prop-3.1-beta-equiv}, we can conclude that there exists a sequence $\rho_k\to0$ as $k\to\infty$ such that
$$
e(b_k)=E_b(u_{b_k})\geq\frac{1}{\epsilon_{b_k}^2}\int_{\Omega_{b_k}}\!|\nabla\varphi_{b_k}|^2\mathrm{d}x
-\frac{\beta}{2\epsilon_{b_k}^2}\int_{\Omega_{b_k}}\!|\varphi_{b_k}|^4\mathrm{d}x
\geq C\epsilon_{b_k}^{-2}e^{-\frac{2}{1+\rho_k}\frac{|z_{b_k}-x_{i_0}|}{\epsilon_{b_k}}}\to\infty\ \mbox{as}\ b_k\to0,
$$
where we have used \eqref{eqn:epsilon-b-0} and the fact that $\{\frac{|z_{b_k}-x_{i_0}|}{\epsilon_{b_k}}\}$ is bounded uniformly as $k\to\infty$, contradicting with Lemma \ref{lem:beta-equiv-inner-upper-bound}. Hence $x_{i_0}\in\Omega$, then we derive that $\Omega_0=\lim_{k\to\infty}\Omega_{b_k}=\mathbb R^2$.
By Lemma \ref{lem:varphi-b-k-Q-convergence}, there holds
$$
\lim_{k\to\infty}\varphi_{b_k}=\frac{Q(|x|)}{\sqrt{\beta^*}}\ \mbox{in}\ H^1(\mathbb R^2).
$$

Next, we claim that $p_{i_0}=p$. Suppose that $p_{i_0}<p$.
Since $\{\frac{|z_{b_k}-x_{i_0}|}{\epsilon_{b_k}}\}$ is bounded uniformly as $k\to\infty$, we can deduce that for sufficiently small $R>0$, there exists $C_0(R)>0$, independent of $b_k$, such that
$$
\begin{aligned}
&\quad\liminf\limits_{\epsilon_{a_k}\rightarrow0}\frac{1}{\epsilon_{b_k}^{p_{i_0}}}
\int_{\mathbb R^2}V(\epsilon_{b_k}x+z_{b_k})|\varphi_{b_k}|^2\mathrm{d}x\\
&\geq C\int_{B_{R}(0)}\liminf\limits_{\epsilon_{b_k}\rightarrow0}\Big|x
+\frac{z_{b_k}-x_{i_0}}{\epsilon_{b_k}}\Big|^{p_{i_0}}
\prod\limits_{j=1,j\neq i_0}^n|\epsilon_{b_k}x+z_{b_k}-x_j|^{p_j}|\varphi_{b_k}|^2\mathrm{d}x
\geq C_0(R),
\end{aligned}
$$
which implies that
$$
\begin{aligned}
e(b)&\geq\frac{b_k}{2\epsilon_{b_k}^4}+C_0(R)\epsilon_{b_k}^{p_{i_0}}
\geq\Big[\frac{1}{2}\Big(\frac{p_{i_0}}{2}\Big)^{\frac{4}{p_{i_0}+4}}
+\Big(\frac{2}{p_{i_0}}\Big)^{\frac{p_{i_0}}{p_{i_0}+4}}\Big]
C_0(R)^{\frac{4}{p_{i_0}+4}}b_k^{\frac{p_{i_0}}{p_{i_0}+4}}\ \mbox{as}\ b_k\to0.
\end{aligned}
$$
This contradicts with Lemma \ref{lem:beta-equiv-inner-upper-bound} due to $p_{i_0}<p$. Thus, we conclude from above that
$z_{b_k}\to x_{i_0}$ as $k\to\infty$ for some $x_{i_0}\in\mathcal Z_1$ and the proof is completed.
\end{proof}

\noindent\textbf{Proof of Theorem \ref{thm:mass-concentration-minimizer-beta-inter}.}
In light of the proof of Lemma \ref{lem:z-b-k-to-x-i-0-mathcal-Z-1}, we know that $\{\frac{|z_{b_k}-x_{i}|}{\epsilon_{b_k}}\}$ is bounded uniformly as $k\to\infty$. Then there exists $y_0\in\mathbb R^2$ such that $\frac{z_{b_k}-x_i}{\epsilon_{b_k}}\to y_0$ with $x_i\in\mathcal Z_1\subset\Omega$, which gives that $\Omega_0=\lim_{k\to\infty}\Omega_{b_k}=\mathbb R^2$. Recalling that $Q$ is a radial decreasing function, it follows from Lemma \ref{lem:varphi-b-k-Q-convergence}
and Fatou's lemma that
\begin{equation}\label{eqn:beta-equiv-inner-V-varphi-0}
\liminf_{\epsilon_{b_k}\to0}\epsilon_{b_k}^{-p}
\int_{\mathbb R^2}V(\epsilon_{b_k}x+z_{b_k})\varphi_{b_k}^2\mathrm{d}x\geq\kappa_i
\int_{\mathbb R^2}|x+y_0|^p\varphi_0^2\mathrm{d}x\geq\kappa_i\int_{\mathbb R^2}|x|^p\varphi_0^2\mathrm{d}x.
\end{equation}
Therefore, we can obtain that
$$
\begin{aligned}
\liminf\limits_{b_k\to0}e(b_k)
&\geq\frac{b_k}{2\epsilon_{b_k}^4}\Big(\int_{\Omega_{b_k}}|\nabla\varphi_{b_k}|^2\mathrm{d}x\Big)^2
+\epsilon_{b_k}^p\kappa_i\int_{\mathbb R^2}|x|^p\varphi_0^2\mathrm{d}x\\
&=\frac{b_k}{2\epsilon_{b_k}^4}+\frac{\epsilon_{b_k}^p\kappa_i}{\beta^*}\int_{\mathbb R^2}|x|^pQ^2\mathrm{d}x\\
&\geq\frac{p+2}{2p}\lambda_i^{\frac{4(p+2)}{p+4}}b_k^{\frac{p}{p+4}},
\end{aligned}
$$
where $i\in\{1,2,\cdots,n\}$. This, together with Lemma \ref{lem:beta-equiv-inner-upper-bound}, indicates that
$$
\lim\limits_{k\to\infty}\frac{e(b_k)}{b_k^{\frac{p}{p+4}}}=\frac{p+2}{2p}\lambda^{\frac{4(p+2)}{p+4}}
\ \mbox{and}\ \lim\limits_{k\to\infty}\frac{\epsilon_{b_k}}{b_k^{\frac{1}{p+4}}}=\lambda^{-{\frac{p+2}{p+4}}}.
$$
Then by \eqref{eqn:beta-equiv-inner-V-varphi-0}, we get
$$
\lim\limits_{k\to\infty}\frac{|z_{b_k}-x_i|}{\epsilon_{b_k}}=0,
$$
where $x_i\in\mathcal Z_1$. The proof is completed.
\qed

\subsection{Mass concentration near the boundary}\label{subsection:beta-equiv-boundary}

In this subsection, we determine that the mass of minimizers for $e(b)$ concentrates near the boundary of $\Omega$ if $\mathcal Z_1=\emptyset$, i.e., Theorem \ref{thm:mass-concentration-minimizer-beta-boundary} is derived. First, we present the upper energy estimate.

\begin{lem}\label{lem:beta-equiv-boundary-energy-estimate}
Suppose that $V(x)$ satisfies \eqref{eqn:V-potential-x-x-i} and $\mathcal Z_1=\emptyset$, it results that
$$
0\leq e(b)\leq\kappa^{\frac{4}{p+4}}2^{-\frac{5p}{p+4}}
\Big(\frac{p+2}{p+4}\Big)^{\frac{4p}{p+4}}
\Big[\Big(\frac{p}{4}\Big)^{\frac{4}{p+4}}+\Big(\frac{4}{p}\Big)^{\frac{p}{p+4}}\Big]
b^{\frac{p}{p+4}}\Big(\ln\frac{2}{b}\Big)^{\frac{4p}{p+4}}\ \mbox{as}\ b\searrow0.
$$
\end{lem}

\begin{proof}
The fact that $\mathcal Z_1=\emptyset$ gives $\mathcal Z_0\neq\emptyset$. Then we can choose $x_i\in\mathcal Z_0$ such that $p_i=p$. From the interior ball conditions of $\Omega$, there exists an open ball $B_R(x_0)$ such that $x_i\in\partial\Omega\cap\partial B_R(x_0)$, where $R>0$ is sufficiently small. Define $R_\tau:=\frac{f(\tau)}{\tau}<R$ for $\tau>0$ large enough, where $0<f(\tau)\in C^2(\mathbb R)$ is chosen to satisfy $\lim_{\tau\rightarrow\infty}f(\tau)=\infty$ and $\lim_{\tau\rightarrow\infty}\frac{f(\tau)}{\tau}=0$. Let $x_\tau:=x_i-(1+\xi(\tau))R_\tau\overrightarrow{n}$, where
$\overrightarrow{n}$ is the unit outward normal vector to $\partial\Omega$ at the point $x_i$ and $\xi(\tau)>0$ is determined such that $\xi(\tau)\rightarrow0$ as $\tau\rightarrow\infty$. Clearly, $B_{(1+\xi(\tau))R_\tau}(x_\tau)\subset\Omega$ and $x_i\in\partial\Omega\cap\partial B_{(1+\xi(\tau))R_\tau}(x_\tau)$ satisfying $\lim_{\tau\rightarrow\infty}x_\tau=x_i$.
Let $\psi_\tau(x)\in C_0^\infty(\mathbb R^2)$ be a nonnegative smooth cut-off function satisfying $\psi_\tau(x)=1$ for $|x|\leq 1$ and $\psi_\tau(x)=0$ for $|x|\geq 1+\xi(\tau)$. Without loss of generality, we assume that $|\nabla\psi_\tau(x)|\leq\frac{M}{\xi(\tau)}$, where $M>0$ is independent of $\tau$.

Take the test function
$$
\Psi_\tau(x)=\frac{A_\tau\tau}{\sqrt{\beta^*}}\psi_\tau\Big(\frac{x-x_\tau}{R_\tau}\Big)Q(\tau(x-x_\tau)),\ x\in\Omega,
$$
where $A_\tau>0$ is chosen such that $\int_\Omega|\Psi_\tau(x)|^2\mathrm{d}x=1$.
Letting
$\Omega_\tau:=\{x\in\mathbb R^2:\big(\frac{x}{\tau}+x_\tau\big)\in\Omega\}$, it follows from the definition of $x_\tau$ that $\Omega_\tau\rightarrow\mathbb R^2$ as $\tau\rightarrow\infty$. Using \eqref{eqn:Q-decay}, we can obtain that
\begin{equation}\label{eqn:A-tau-beta-equiv-boundary}
1\leq A_\tau^2\leq1+Ce^{-2\tau R_\tau}\ \mbox{as}\ \tau\rightarrow\infty.
\end{equation}
Furthermore, we also derive that
$$
\begin{aligned}
&\quad\ \int_{\Omega}|\nabla\Psi_\tau(x)|^2\mathrm{d}x\\
&=\int_{\mathbb R^2}\frac{A_\tau^2\tau^2}{\beta^*R_\tau^2}
\Big|\nabla\psi_\tau\Big(\frac{x-x_\tau}{R_\tau}\Big)\Big|^2Q^2(\tau(x-x_\tau))\mathrm{d}x
+\int_{\mathbb R^2}\frac{A_\tau^2\tau^4}{\beta^*}\psi_\tau^2\Big(\frac{x-x_\tau}{R_\tau}\Big)|\nabla Q(\tau(x-x_\tau))|^2\mathrm{d}x\\
&\quad+2\int_{\mathbb R^2}\frac{A_\tau^2\tau^3}{\beta^*R_\tau}\psi_\tau\Big(\frac{x-x_\tau}{R_\tau}\Big)Q(\tau(x-x_\tau))
\nabla\psi_\tau\Big(\frac{x-x_\tau}{R_\tau}\Big)\cdot\nabla Q(\tau(x-x_\tau))\mathrm{d}x\\
&\leq\int_{\mathbb R^2}\frac{A_\tau^2\tau^2}{\beta^*}|\nabla Q(x)|^2\mathrm{d}x
+C\Big(\frac{\tau}{R_\tau\xi(\tau)}+\frac{1}{\xi^2(\tau)R_\tau^2}\Big)e^{-2\tau R_\tau}\ \mbox{as}\ \tau\rightarrow\infty.
\end{aligned}
$$
The above inequality follows the estimates below:
$$
\begin{aligned}
&\quad\ \int_{\mathbb R^2}\frac{A_\tau^2\tau^2}{\beta^*R_\tau^2}
\Big|\nabla\psi_\tau\Big(\frac{x-x_\tau}{R_\tau}\Big)\Big|^2Q^2(\tau(x-x_\tau))\mathrm{d}x\\
&=\int_{1\leq|x|\leq1+\xi(\tau)}\frac{A_\tau^2\tau^2}{\beta^*}
|\nabla\psi_\tau(x)|^2Q^2(\tau R_\tau x)\mathrm{d}x
\leq\frac{C}{\xi^2(\tau)R_\tau^2}e^{-2\tau R_\tau}\ \mbox{as}\ \tau\rightarrow\infty,
\end{aligned}
$$
$$
\begin{aligned}
&\quad\ \int_{\mathbb R^2}\frac{A_\tau^2\tau^4}{\beta^*}\psi_\tau^2\Big(\frac{x-x_\tau}{R_\tau}\Big)|\nabla Q(\tau(x-x_\tau))|^2\mathrm{d}x\\
&=\int_{\mathbb R^2}\frac{A_\tau^2\tau^2}{\beta^*}\Big(\psi_\tau^2\Big(\frac{x}{\tau R_\tau}\Big)-1\Big)
|\nabla Q|^2\mathrm{d}x
+\int_{\mathbb R^2}\frac{A_\tau^2\tau^2}{\beta^*}|\nabla Q|^2\mathrm{d}x
\leq\int_{\mathbb R^2}\frac{A_\tau^2\tau^2}{\beta^*}|\nabla Q|^2\mathrm{d}x\ \mbox{as}\ \tau\rightarrow\infty
\end{aligned}
$$
and
$$
\begin{aligned}
&\quad\ 2\int_{\mathbb R^2}\frac{A_\tau^2\tau^3}{\beta^*R_\tau}\psi_\tau\Big(\frac{x-x_\tau}{R_\tau}\Big)Q(\tau(x-x_\tau))
\nabla\psi_\tau\Big(\frac{x-x_\tau}{R_\tau}\Big)\cdot\nabla Q(\tau(x-x_\tau))\mathrm{d}x\\
&=2\int_{\mathbb R^2}\frac{A_\tau^2\tau^3R_\tau}{\beta^*}\psi_\tau(x)Q(\tau R_\tau x)
\nabla\psi_\tau(x)\cdot\nabla Q(\tau R_\tau x)\mathrm{d}x\\
&\leq\frac{2MA_\tau^2\tau^3R_\tau}{\xi(\tau)\beta^*}
\int_{1\leq|x|\leq1+\xi(\tau)}Q(\tau R_\tau x)|\nabla Q(\tau R_\tau x)|\mathrm{d}x
\leq\frac{C\tau}{\xi(\tau)R_\tau}e^{-2\tau R_\tau}\ \mbox{as}\ \tau\rightarrow\infty.
\end{aligned}
$$
Then choosing a suitable function $\xi(\tau)>0$ so that $\frac{\tau}{R_\tau\xi(\tau)}=o(\tau^2)$ and $\frac{1}{R_\tau^2\xi^2(\tau)}=o(\tau^2)$ as $\tau\rightarrow\infty$,
we can see from \eqref{eqn:A-tau-beta-equiv-boundary} that
$$
\int_\Omega|\nabla\Psi_\tau(x)|^2\mathrm{d}x\leq\tau^2+C\tau^2e^{-2\tau R_\tau}\ \mbox{as}\ \tau\rightarrow\infty
$$
and
$$
\left(\int_\Omega|\nabla\Psi_\tau|^2\mathrm{d}x\right)^2\leq\tau^4+C\tau^4e^{-2\tau R\tau}\ \mbox{as}\ \tau\rightarrow\infty.
$$
Similarly, we also have
$$
\begin{aligned}
\int_\Omega|\Psi_\tau(x)|^4\mathrm{d}x
&=\int_{\mathbb R^2}\frac{A_\tau^4\tau^2}{(\beta^*)^2}\psi_\tau^4\Big(\frac{x}{\tau R_\tau}\big)Q^4(x)\mathrm{d}x\\
&\geq\int_{|x|\leq\tau R_\tau}\frac{A_\tau^4\tau^2}{(\beta^*)^2}Q^4(x)\mathrm{d}x\\
&=\int_{\mathbb R^2}\!\!\frac{A_\tau^4\tau^2}{(\beta^*)^2}Q^4(x)\mathrm{d}x
-\int_{|x|\geq\tau R_\tau}\!\!\frac{A_\tau^4\tau^2}{(\beta^*)^2}Q^4(x)\mathrm{d}x\\
&\geq\frac{2\tau^2}{\beta^*}-o(\tau^2e^{-2\tau R_\tau})\ \mbox{as}\ \tau\rightarrow\infty.
\end{aligned}
$$
Now, we choose $f(\tau)=\tau R_\tau=\frac{p+2}{2}\ln\tau$, then we get
$C\tau^2e^{-2f(\tau)}=C\tau^{-p}=o\Big(\Big(\frac{\ln\tau}{\tau}\Big)^p\Big)\ \mbox{as}\ \tau\rightarrow\infty$.
From above estimates, we can obtain that
$$
\int_\Omega|\nabla\Psi_\tau|^2\mathrm{d}x+\frac{b}{2}\left(\int_\Omega|\nabla\Psi_\tau|^2\mathrm{d}x\right)^2
-\frac{\beta}{2}\int_\Omega|\Psi_\tau|^{4}\mathrm{d}x
\leq\frac{b}{2}\tau^4+o\Big(\Big(\frac{\ln\tau}{\tau}\Big)^p\Big)\ \mbox{as}\ \tau\rightarrow\infty.
$$
Since $V(x)$ satisfies \eqref{eqn:V-potential-x-x-i} and $\Omega$ is a bounded domain, there is $M>0$ such that
$$
V(x)\leq M|x-x_i|^p\ \mbox{in}\ \Omega,
$$
which states that
\begin{equation}\label{eqn:V-M-x-tau-x-i-p}
V\Big(\frac{x}{\tau}+x_\tau\Big)\leq M\Big|\frac{x}{\tau}+x_\tau-x_i\Big|^p\ \mbox{for\ any}\ x\in\Omega_\tau.
\end{equation}
Therefore,
$$
\begin{aligned}
\int_{\Omega}V(x)|\Psi_\tau(x)|^2\mathrm{d}x
&=\int_{\Omega_\tau}V\Big(\frac{x}{\tau}+x_\tau\Big)\frac{A_\tau^2}{\beta^*}\psi_\tau^2\Big(\frac{2x}{(p+2)\ln\tau}\Big)
Q^2(x)\mathrm{d}x\\
&=\frac{A_\tau^2}{\beta^*}\int_{B_{\sqrt{\ln\tau}}(0)}\!\!\!V\Big(\frac{x}{\tau}+x_\tau\Big)Q^2(x)\mathrm{d}x\\
&\quad\ +\frac{A_\tau^2}{\beta^*}\int_{B_{\frac{(1+\xi(\tau))(p+2)\ln\tau}{2}}(0)/B_{\sqrt{\ln\tau}}(0)}
\!\!\!V\Big(\frac{x}{\tau}+x_\tau\Big)\psi_\tau^2\Big(\frac{2x}{(p+2)\ln\tau}\Big)Q^2(x)\mathrm{d}x\\
&=\frac{A_\tau^2}{\beta^*}\int_{B_{\sqrt{\ln\tau}}(0)}\!\!\!V\Big(\frac{x}{\tau}+x_\tau\Big)Q^2(x)\mathrm{d}x
+o\Big(\Big(\frac{\ln\tau}{\tau}\Big)^p\Big)\\
&=\kappa_i\Big(\frac{p+2}{2}\Big)^p\Big(\frac{\ln\tau}{\tau}\Big)^p+o\Big(\Big(\frac{\ln\tau}{\tau}\Big)^p\Big)
\ \mbox{as}\ \tau\rightarrow\infty,
\end{aligned}
$$
where we have used \eqref{eqn:V-M-x-tau-x-i-p} and $\frac{\sqrt{\ln\tau}}{\tau}=o(R_\tau)=o(|x_\tau-x_i|)$ as $\tau\rightarrow\infty$. Thus, we conclude that
$$
e(b)\leq E_b(\Psi_\tau)\leq\frac{b}{2}\tau^4
+\kappa_i\Big(\frac{p+2}{2}\Big)^p\Big(\frac{\ln\tau}{\tau}\Big)^p+o\Big(\Big(\frac{\ln\tau}{\tau}\Big)^p\Big)
\ \mbox{as}\ \tau\rightarrow\infty.
$$
According to Remark \ref{rem:function-h-t-analysis} below, we can take
$$
\tau=\Big(\frac{p\kappa_i}{2^{p+1}}\Big)^{\frac{1}{p+4}}\Big(\frac{p+2}{p+4}\Big)^{\frac{p}{p+4}}
b^{-\frac{1}{p+4}}(\ln\frac{2}{b})^{\frac{p}{p+4}}
$$
such that $\tau\to\infty$ as $b\searrow0$. Then
$$
0\leq e(b)\leq \kappa_i^{\frac{4}{p+4}}2^{-\frac{5p}{p+4}}
\Big(\frac{p+2}{p+4}\Big)^{\frac{4p}{p+4}}
\Big[\Big(\frac{p}{4}\Big)^{\frac{4}{p+4}}+\Big(\frac{4}{p}\Big)^{\frac{p}{p+4}}\Big]
b^{\frac{p}{p+4}}\Big(\ln\frac{2}{b}\Big)^{\frac{4p}{p+4}}\ \mbox{as}\ b\searrow0.
$$
The proof is completed.
\end{proof}

\begin{rem}\label{rem:function-h-t-analysis}
For $a\searrow0$ and $\gamma,p>0$, define
$$
h(t)=at^{-4}+\gamma t^p\Big(\ln\frac{1}{t}\Big)^p,\ t\in(0,\frac{1}{e+1}).
$$
By a direct computation, we have that $h'(t)=0$ is equivalent to
$$
t^{p+4}\big[\Big(\ln\frac{1}{t}\Big)^p-\Big(\ln\frac{1}{t}\Big)^{p-1}\Big]=\frac{4a}{\gamma p},
$$
which, together with infinitesimal analysis, states that its zero
$$
t_0=(1+o(1))\Big(\frac{4}{p\gamma}\Big)^{\frac{1}{p+4}}(p+4)^{\frac{p}{p+4}}a^{\frac{1}{p+4}}
\Big(\ln\frac{1}{a}\Big)^{\frac{-p}{p+4}}\ \mbox{as}\ a\searrow0.
$$
Further, we deduce that $h(t)$ admits its unique minimum point at $t_0\in(0,\frac{1}{e+1})$ and
$$
h(t_0)=(1+o(1))\gamma^{\frac{4}{p+4}}\Big(\frac{1}{p+4}\Big)^{\frac{4p}{p+4}}
\Big[\Big(\frac{p}{4}\Big)^{\frac{4}{p+4}}+\Big(\frac{4}{p}\Big)^{\frac{p}{p+4}}\Big]
a^{\frac{p}{p+4}}\Big(\ln\frac{1}{a}\Big)^{\frac{4p}{p+4}}\ \mbox{as}\ a\searrow0.
$$
\end{rem}

\begin{lem}
Suppose $V(x)$ satisfies \eqref{eqn:V-potential-x-x-i} and $\mathcal Z_1=\emptyset$, then
for the subsequence $\{b_k\}$ given in Lemma \ref{lem:varphi-b-k-Q-convergence}, there holds
$z_{b_k}\rightarrow x_i$ as $k\rightarrow\infty$ for some $x_i\in\mathcal Z_0$.
\end{lem}

\begin{proof}
In view of Lemma \ref{lem:varphi-b-k-Q-convergence}, it results that $z_{b_k}\rightarrow x_{i_0}$ as $k\rightarrow\infty$, where $x_{i_0}\in\bar\Omega$ satisfies $V(x_{i_0})=0$. Now, we first consider $x_{i_0}\in\Omega$ with $0<p_{i_0}<p$, or $x_{i_0}\in\partial\Omega$ satisfying $0<p_{i_0}<p$ and $\frac{|z_{b_k}-x_{i_0}|}{\epsilon_{b_k}}\rightarrow\infty$ as $k\rightarrow\infty$. For two cases, we can obtain that $\Omega_0=\lim_{k\rightarrow\infty}\Omega_{b_k}=\mathbb R^2$.
Then it follows from \eqref{eqn:omega-b-2-beta-strict-0-equiv} that for sufficiently small $R>0$,
there exists $C_0(R)>0$, independent of $b_k$, such that
$$
\begin{aligned}
&\quad\liminf\limits_{\epsilon_{b_k}\rightarrow0}\frac{1}{\epsilon_{b_k}^{p_{i_0}}}
\int_{\mathbb R^2}V(\epsilon_{b_k}x+z_{b_k})|\varphi_{b_k}|^2\mathrm{d}x\\
&\geq C\int_{B_{R}(0)}\liminf\limits_{\epsilon_{b_k}\rightarrow0}\Big|x
+\frac{z_{b_k}-x_{i_0}}{\epsilon_{b_k}}\Big|^{p_{i_0}}
\prod\limits_{j=1,j\neq i_0}^n|\epsilon_{b_k}x+z_{b_k}-x_j|^{p_j}|\varphi_{b_k}|^2\mathrm{d}x
\geq C_0(R).
\end{aligned}
$$
This yields that
$$
\begin{aligned}
e(b_k)&\geq\frac{b_k}{2\epsilon_{b_k}^4}+C_0(R)\epsilon_{b_k}^{p_{i_0}}
\geq\Big[\frac{1}{2}\Big(\frac{p_{i_0}}{2}\Big)^{\frac{4}{p_{i_0}+4}}
+\Big(\frac{2}{p_{i_0}}\Big)^{\frac{p_{i_0}}{p_{i_0}+4}}\Big]
C_0(R)^{\frac{4}{p_{i_0}+4}}b_k^{\frac{p_{i_0}}{p_{i_0}+4}}\ \mbox{as}\ b_k\to0,
\end{aligned}
$$
which contradicts with Lemma \ref{lem:beta-equiv-boundary-energy-estimate} due to $p_{i_0}<p$.

Next, we rule out the case where $x_{i_0}\in\partial\Omega$ satisfies $0<p_{i_0}<p$ and $\{\frac{|z_{b_k}-x_{i_0}|}{\epsilon_{b_k}}\}$ is uniformly bounded as $k\rightarrow\infty$. By Lemma \ref{lem:similar-to-Guo-Prop-3.1-beta-equiv}, we deduce that there exists a sequence $\rho_k\to0$ as $k\to\infty$ such that
$$
e(b_k)=E_b(u_{b_k})\geq\frac{1}{\epsilon_{b_k}^2}\int_{\Omega_{b_k}}\!|\nabla\varphi_{b_k}|^2\mathrm{d}x
-\frac{\beta}{2\epsilon_{b_k}^2}\int_{\Omega_{b_k}}\!|\varphi_{b_k}|^4\mathrm{d}x
\geq C\epsilon_{b_k}^{-2}e^{-\frac{2}{1+\rho_k}\frac{|z_{b_k}-x_{i_0}|}{\epsilon_{b_k}}}\to\infty\ \mbox{as}\ k\to\infty,
$$
which also contradicts Lemma \ref{lem:beta-equiv-boundary-energy-estimate}. Hence, the proof is completed.
\end{proof}

\begin{lem}\label{lem:z-b-k-x-i-epsilon-b-k-infty}
Suppose $V(x)$ satisfies \eqref{eqn:V-potential-x-x-i} and $\mathcal Z_1=\emptyset$, we have
$$
\limsup\limits_{k\rightarrow\infty}\frac{|z_{b_k}-x_i|}{\epsilon_{b_k}|\ln\epsilon_{b_k}|}<\infty\ \mbox{and}\
\liminf\limits_{k\rightarrow\infty}\frac{|z_{b_k}-x_i|}{\epsilon_{b_k}|\ln\epsilon_{b_k}|}\geq\frac{p+2}{2}.
$$
\end{lem}

\begin{proof}
We prove the first one by contradiction. Suppose that there exists a subsequence of $\{b_k\}$ such that $\frac{|z_{b_k}-x_i|}{\epsilon_{b_k}|\ln\epsilon_{b_k}|}\rightarrow\infty$ as $k\rightarrow\infty$, then we can derive from \eqref{eqn:omega-b-2-beta-strict-0-equiv} that for any large constant $M>0$,
$$
\begin{aligned}
&\quad\liminf\limits_{\epsilon_{b_k}\rightarrow0}\frac{1}{(\epsilon_{b_k}|\ln\epsilon_{b_k}|)^p}
\int_{B_{2R}(0)}V(\epsilon_{b_k}x+z_{b_k})|\varphi_{b_k}|^2\mathrm{d}x\\
&\geq C\int_{B_{2R}(0)\cap\Omega_{b_k}}\liminf\limits_{\epsilon_{b_k}\rightarrow0}
\Big|\frac{x}{|\ln\epsilon_{b_k}|}+\frac{z_{b_k}-x_{i}}{\epsilon_{b_k}|\ln\epsilon_{b_k}|}\Big|^{p}
\prod\limits_{j=1,j\neq i}^n|\epsilon_{b_k}x+z_{b_k}-x_j|^{p_j}|\varphi_{b_k}|^2\mathrm{d}x\\
&\geq M.
\end{aligned}
$$
Then
$$
e(b_k)\geq\frac{b_k}{2\epsilon_{b_k}^4}+M(\epsilon_{b_k}|\ln\epsilon_{b_k}|)^{p}
\geq CM^{\frac{4}{p+4}}b_k^{\frac{p}{p+4}}\Big(\ln\frac{2}{b_k}\Big)^{\frac{4p}{p+4}},
$$
which is a contradiction with Lemma \ref{lem:beta-equiv-boundary-energy-estimate}.

Next, we proceed the remain one. Suppose that there exists a subsequence of $\{b_k\}$ such that $\lim_{k\rightarrow\infty}\frac{|z_{b_k}-x_i|}{\epsilon_{b_k}|\ln\epsilon_{a_k}|}=\sigma<\frac{p+2}{2}$. By Lemma \ref{lem:similar-to-Guo-Prop-3.1-beta-equiv}, we can see that as $k\to\infty$,
\begin{equation}\label{eqn:to-prove-z-b-k-x-i-epsilon-b-k-away-zero}
\begin{aligned}
e(b_k)=E_b(u_{b_k})
&\geq\frac{1}{\epsilon_{b_k}^2}\int_{\Omega_{b_k}}|\nabla\varphi_{b_k}|^2\mathrm{d}x
+\frac{b}{2\epsilon_{b_k}^4}\Big(\int_{\Omega_{b_k}}|\nabla\varphi_{b_k}|^2\mathrm{d}x\Big)^2
-\frac{\beta}{2\epsilon_{b_k}^2}\int_{\Omega_{b_k}}|\varphi_{b_k}|^4\mathrm{d}x\\
&\geq\frac{b}{2}\epsilon_{b_k}^{-4}+C\epsilon_{b_k}^{-2}e^{-\frac{2}{1+\rho_k}\frac{|z_{b_k}-x_{i}|}{\epsilon_{b_k}}}
=\frac{b}{2}\epsilon_{b_k}^{-4}+C\epsilon_{b_k}^{\frac{2\sigma}{1+\rho_k}-2}.
\end{aligned}
\end{equation}
Provided that $\frac{\sigma}{1+\rho_k}\leq1$ for sufficiently large $k$, by \eqref{eqn:to-prove-z-b-k-x-i-epsilon-b-k-away-zero}, we can directly get $e(b_k)\geq C$, which is impossible due to
Lemma \ref{lem:beta-equiv-boundary-energy-estimate}. On the other hand, we consider $\frac{\sigma}{1+\rho_k}>1$ for sufficiently large $k$. Then there exists a sufficiently large $k_0>0$ such that
$$
\frac{\sigma}{1+\rho_k}<\frac{\sigma}{2}+\frac{p+2}{4}\ \mbox{for}\ k>k_0.
$$
Denote $0<q:=\frac{2\sigma+p-2}{2}<p$, then we can obtain that
$$
0<1-\frac{2}{\frac{\sigma}{1+\rho_k}+1}<1-\frac{2}{\frac{\sigma}{2}+\frac{p+2}{4}+1}=\frac{q}{q+4}<\frac{p}{p+4}
\ \mbox{for}\ k>k_0.
$$
By \eqref{eqn:to-prove-z-b-k-x-i-epsilon-b-k-away-zero}, we can get
$$
e(b_k)\geq Cb_k^{1-\frac{2}{\frac{\sigma}{1+\rho_k}+1}}>Cb_k^{\frac{q}{q+4}}\ \mbox{for}\ k>k_0,
$$
contradicting with Lemma \ref{lem:beta-equiv-boundary-energy-estimate}. The proof is completed.
\end{proof}

\noindent\textbf{Proof of Theorem \ref{thm:mass-concentration-minimizer-beta-boundary}.}
In view of Lemma \ref{lem:z-b-k-x-i-epsilon-b-k-infty}, we have $\frac{|z_{b_k}-x_i|}{\epsilon_{b_k}}\to\infty$ as $k\to\infty$, which gives $\Omega_0=\lim_{k\to\infty}\Omega_{b_k}=\mathbb R^2$. Then it follows from Lemma \ref{lem:varphi-b-k-Q-convergence} that
$$
\lim\limits_{k\rightarrow\infty}\varphi_{b_k}(x)=\varphi_0=\frac{Q(|x|)}{\sqrt{\beta^*}}\ \mbox{in}\ H^1(\mathbb R^2).
$$
Hence $\varphi_0$ admits the exponential decay as $|x|\to\infty$, and by comparison principle, $\varphi_{b_k}$ admits the exponential decay near $|x|\to\infty$ as $k\to\infty$. Together with Lemma \ref{lem:z-b-k-x-i-epsilon-b-k-infty}, we can deduce that
\begin{equation}\label{eqn:in-order-to-V-estimate-ln-varepsilon}
\int_{\Omega_{b_k}}\Big|\frac{x}{|\ln\epsilon_{a_k}|}+\frac{z_{b_k}-x_i}{\epsilon_{b_k}|\ln\epsilon_{b_k}|}\Big|^p
\Big|\varphi_{b_k}^2-\frac{Q^2(x)}{\beta^*}\Big|\mathrm{d}x
\rightarrow0\ \mbox{as}\ k\rightarrow\infty.
\end{equation}
By the definition of $\Omega_{b_k}$, we know that $\{\epsilon_{b_k}x+z_{b_k}\}$ is bounded uniformly in $k$. Then there exists $M>0$ such that
\begin{equation}\label{eqn:proof-of-Theorem-1.2-potential}
V(\epsilon_{b_k}x+z_{b_k})\leq M|\epsilon_{b_k}x+z_{b_k}-x_i|^p\ \mbox{for}\ x\in\Omega_{b_k}.
\end{equation}
Based on \eqref{eqn:in-order-to-V-estimate-ln-varepsilon} and \eqref{eqn:proof-of-Theorem-1.2-potential}, we can estimate that
$$
\begin{aligned}
&\quad\ \Big|\int_{\Omega_{b_k}}V(\epsilon_{b_k}x+z_{b_k})\varphi_{b_k}^2\mathrm{d}x
-\int_{\Omega_{b_k}}V(\epsilon_{b_k}x+z_{b_k})\frac{Q^2(x)}{\beta^*}\mathrm{d}x\Big|\\
&\leq M\int_{\Omega_{b_k}}|\epsilon_{b_k}x+z_{b_k}-x_i|^p\Big|\varphi_{b_k}^2-\frac{Q^2(x)}{\beta^*}\Big|\mathrm{d}x\\
&=M\varepsilon_{b_k}^p|\ln\epsilon_{b_k}|^p
\int_{\Omega_{b_k}}\Big|\frac{x}{|\ln\epsilon_{b_k}|}
+\frac{z_{b_k}-x_i}{\epsilon_{b_k}|\ln\epsilon_{b_k}|}\Big|^p
\Big|\varphi_{b_k}^2-\frac{Q^2(x)}{\beta^*}\Big|\mathrm{d}x\\
&=o(\epsilon_{b_k}^p|\ln\epsilon_{b_k}|^p)\ \mbox{as}\ k\rightarrow\infty.
\end{aligned}
$$
Therefore,
\begin{equation}\label{eqn:boundary-lower-bound}
\begin{aligned}
e(b_k)&\geq\frac{b_k}{2\epsilon_{b_k}^4}\Big(\int_{\Omega_{b_k}}|\nabla\varphi_{b_k}|^2\mathrm{d}x\Big)^2
+\int_{\Omega_{b_k}}V(\epsilon_{b_k}x+z_{b_k})\frac{Q^2(x)}{\beta^*}\mathrm{d}x
+o(\epsilon_{b_k}^p|\ln\epsilon_{b_k}|^p)\\
&=\frac{b_k}{2\epsilon_{b_k}^4}+\kappa_i\varrho^p\epsilon_{b_k}^p|\ln\epsilon_{b_k}|^p
+o(\epsilon_{b_k}^p|\ln\epsilon_{b_k}|^p)\ \mbox{as}\ k\rightarrow\infty,
\end{aligned}
\end{equation}
where $\varrho:=\lim_{k\rightarrow\infty}\frac{|z_{b_k}-x_i|}{\epsilon_{b_k}|\ln\epsilon_{b_k}|}\geq\frac{p+2}{2}$
due to Lemma \ref{lem:z-b-k-x-i-epsilon-b-k-infty}.  By \eqref{eqn:boundary-lower-bound} and Remark \ref{rem:function-h-t-analysis}, we check that
$$
e(b_k)\geq\kappa^{\frac{4}{p+4}}2^{-\frac{p}{p+4}}\Big(\frac{\varrho}{p+4}\Big)^{\frac{4p}{p+4}}
\Big[\Big(\frac{p}{4}\Big)^{\frac{4}{p+4}}+\Big(\frac{4}{p}\Big)^{\frac{p}{p+4}}\Big]
b_k^{\frac{p}{p+4}}\Big(\ln\frac{2}{b_k}\Big)^{\frac{4p}{p+4}}\ \mbox{as}\ b_k\searrow0.
$$
Combining with Lemma \ref{lem:beta-equiv-boundary-energy-estimate}, it results that $\varrho\leq\frac{p+2}{2}$. Then $\varrho=\frac{p+2}{2}$, i.e.,
$$
\lim\limits_{k\rightarrow\infty}\frac{|z_{b_k}-x_i|}
{\epsilon_{b_k}|\ln\epsilon_{b_k}|}=\frac{p+2}{2}.
$$
Moreover, we can conclude that
$$
\lim\limits_{k\to\infty}\frac{e(b_k)}{b_k^{\frac{p}{p+4}}\Big(\ln\frac{2}{b_k}\Big)^{\frac{4p}{p+4}}}=\kappa^{\frac{4}{p+4}}2^{-\frac{5p}{p+4}}
\Big(\frac{p+2}{p+4}\Big)^{\frac{4p}{p+4}}
\Big[\Big(\frac{p}{4}\Big)^{\frac{4}{p+4}}+\Big(\frac{4}{p}\Big)^{\frac{p}{p+4}}\Big]
$$
and
$$
\lim\limits_{k\to\infty}\frac{\epsilon_{b_k}}{b_k^{\frac{1}{p+4}}\Big(\ln\frac{2}{b_k}\Big)^{\frac{p}{p+4}}}
=\Big(\frac{2^{p+1}}{pk}\Big)^{\frac{1}{p+4}}
\Big(\frac{p+4}{p+2}\Big)^{\frac{p}{p+4}}.
$$
The proof is complete.
\qed

\section{Mass concentration of minimizers for $\beta>\beta^*$}\label{sec:mass-concentration-bar-e-beta-b}

In this section, under the assumption $\beta>\beta^*$, we consider $\mathcal Z_1\neq\emptyset$ or $\mathcal Z_1=\emptyset$, and prove Theorems \ref{thm:mass-concentration-minimizer-beta-strict} and \ref{thm:mass-concentration-minimizer-beta-strict-boundary}. Recalling Proposition \ref{pro:bar-e-b-explicit-form-bar-u-b}, it presents the properties of auxiliary problem $\bar e(b)$.
In what follows, we establish some prior estimates of minimizers.

\begin{lem}\label{lem:nabla-u-b-r-b-mass-critical-term-estimate}
Suppose that $V(x)$ satisfies $(V_1)$ and let $u_b$ be a nonnegative minimizer of $e(b)$, there hold that
$$
\frac{\int_{\Omega}|\nabla u_b|^2\mathrm{d}x}{r_b}\rightarrow1\ \mbox{and}\
\frac{\int_\Omega|u_b|^4\mathrm{d}x}{r_b}\rightarrow\frac{2}{\beta^*}\ \mbox{as}\ b\searrow0.
$$
\end{lem}

\begin{proof}
Suppose that there exists some $\delta$ such that
\begin{equation}\label{eqn:limit-delta}
\frac{\int_{\Omega}|\nabla u_b|^2\mathrm{d}x}{r_b}\rightarrow\delta\ \mbox{as}\ b\searrow0.
\end{equation}
We first prove that it is always impossible for $\delta\in[0,1)$ and $\delta>1$ by contradiction. If $\delta\in[0,1)$, there exists $\varepsilon>0$ such that $\xi:=\delta+\varepsilon<1$ and
$\frac{\int_{\Omega}|\nabla u_b|^2\mathrm{d}x}{r_b}\leq\xi$ as $b\searrow0$.
Now, define
$$
g(r)=\frac{b}{2}r^2-\frac{\beta-\beta^*}{\beta^*}r,\ r\in[0,\infty).
$$
Then we know that $g(r)$ reaches its unique global minimum at $r_b=\frac{\beta-\beta^*}{\beta^*b}$ and
$
g(r_b)=-\frac{1}{2b}\Big(\frac{\beta-\beta^*}{\beta^*}\Big)^2=\bar e(b).
$
By \eqref{eqn:bar-e-b-value} and Theorem \ref{thm:existence-nonexistence-e-0}-($iii$), we can deduce that
$$
0>e(b)=E_b(u_b)\geq g\Big(\int_\Omega|\nabla u_b|^2\mathrm{d}x\Big)\geq g(\xi r_b)\geq g(r_b)=\bar e(b)\ \mbox{as}\ b\searrow0.
$$
Thus,
\begin{equation}\label{eqn:e-b-g-y-b-leq-strictly-1}
\lim\limits_{b\searrow0}\frac{e(b)}{g(r_b)}\leq\lim\limits_{b\searrow0}\frac{g(\xi r_b)}{g(r_b)}
=\lim\limits_{b\searrow0}\frac{\frac{b}{2}\xi^2r_b^2-\frac{\beta-\beta^*}{\beta^*}\xi r_b}{\frac{b}{2}r_b^2-\frac{\beta-\beta^*}{\beta^*}r_b}=-\xi^2+2\xi\in(0,1)\ \mbox{for\ all}\ \xi\in[0,1).
\end{equation}
On the other hand, using \eqref{eqn:bar-e-b-value} and Theorem \ref{thm:existence-nonexistence-e-0}-($iii$) again, we have
$$
\lim\limits_{b\searrow0}\frac{e(b)}{g(r_b)}=\lim\limits_{b\searrow0}\frac{\bar e(b)+o(1)}{\bar e(b)}=1,
$$
which contradicts with \eqref{eqn:e-b-g-y-b-leq-strictly-1}. Similarly, we also get a contradiction for $\delta>1$.
Moreover, if the limit in \eqref{eqn:limit-delta} does not exist, there exist two constants $m_1\in(0,\infty]$ and $m_2\in[0,\infty)$ such that
$$
\limsup\limits_{b\searrow0}\frac{\int_\Omega|\nabla u_b|^2\mathrm{d}x}{r_b}=m_1\ \mbox{and}\
\liminf\limits_{b\searrow0}\frac{\int_\Omega|\nabla u_b|^2\mathrm{d}x}{r_b}=m_2.
$$
Due to $m_1\neq m_2$, we know $m_1\neq 1$ or $m_2\neq 1$. Then we can obtain a contradiction by repeating above argument, and the first result is proved.

Finally, we verify the remaining one.
It is easy to check that
\begin{equation}\label{eqn:int-V-u-b-0}
\int_{\Omega}V(x)u_b^2\mathrm{d}x\rightarrow0
\ \mbox{as}\ b\searrow0.
\end{equation}
In fact,
define $\tilde u_b\equiv u_b$ for $x\in\Omega$ and $\tilde u_b\equiv0$ for $x\in\mathbb R^2\backslash\Omega$, then Theorem \ref{thm:existence-nonexistence-e-0}-$(iii)$ implies
$$
\int_{\Omega}V(x)u_b^2\mathrm{d}x=E_b(u_b)-\bar E_b(\tilde u_b)
\leq e(b)-\bar e(b)\rightarrow0\ \mbox{as}\ b\searrow0.
$$
Since $u_b$ is a minimizer of $e(b)$, by applying \eqref{eqn:bar-e-b-value}-\eqref{eqn:bar-u-b-form-r-b-form} and \eqref{eqn:int-V-u-b-0}, there is
$$
\frac{\beta\int_{\Omega}|u_b|^4\mathrm{d}x}{2r_b}
=\frac{\int_\Omega(|\nabla u_b|^2+V(x)u_b^2)\mathrm{d}x}{r_b}
+\frac{b\Big(\int_\Omega|\nabla u_b|^2\mathbb R^2\Big)^2}{2r_b}
-\frac{e(b)}{r_b}\to\frac{\beta}{\beta^*}\ \mbox{as}\ b\searrow0,
$$
i.e.,
$$
\frac{\int_\Omega|u_b|^4\mathrm{d}x}{r_b}\to\frac{2}{\beta^*}\ \mbox{as}\ b\searrow0.
$$
The proof is completed.
\end{proof}

\begin{lem}\label{lem:w-b-k-convergence-property}
Suppose that $V$ satisfies $(V_1)$. Let $u_b$ be a nonnegative minimizer of $e(b)$ and $z_b$ be a maximum point of $u_b$ in $\Omega$. Define
$$
w_b(x):=\varepsilon_bu_b(\varepsilon_bx+z_b),\ x\in\Omega_b,
$$
where $\varepsilon_b=r_b^{-\frac{1}{2}}$ and $\Omega_b=\{x\in\mathbb R^2:(\varepsilon_bx+z_b)\in\Omega\}$. Then we have
\begin{enumerate}
\item [$(i)$] There exist $R>0$ and $\theta>0$ such that
\begin{equation}\label{eqn:w-b-2-theta-strict-0}
\liminf\limits_{b\searrow0}\int_{B_{2R}(0)\cap\Omega_b}|w_b|^2\mathrm{d}x\geq\theta>0.
\end{equation}
\item [$(ii)$] For any sequence $\{b_k\}$ satisfying $b_k\searrow0$ as $k\rightarrow\infty$, there exists a subsequence, still denoted by $\{b_k\}$, such that $\lim_{k\rightarrow\infty}z_{b_k}=x_{i_0}\in\bar{\Omega}$ satisfying $V(x_{i_0})=0$. Moreover, if
$$
\Omega_0=\lim\limits_{k\rightarrow\infty}\Omega_{b_k}=\mathbb R^2,
$$
then $u_{b_k}$ has a unique maximum point $z_{b_k}$ and
\begin{equation}\label{eqn:beta-strict-w-b-k-convergence-Q}
\lim\limits_{k\rightarrow\infty}w_{b_k}(x)=\frac{Q(|x|)}{\sqrt{\beta^*}}
\ \mbox{in}\ H^1(\mathbb R^2).
\end{equation}
\end{enumerate}
\end{lem}

\begin{proof}
Owing to the definition of $w_b$, $\int_{\Omega_b}|w_b|^2\mathrm{d}x=1$. It follows from Lemma \ref{lem:nabla-u-b-r-b-mass-critical-term-estimate} that
\begin{equation}\label{eqn:varepsilon-2-u-b-2-u-b-4-convergence}
\int_{\Omega_b}|\nabla w_b|^2\mathrm{d}x=\varepsilon_b^2\int_{\Omega}|\nabla u_b|^2\mathrm{d}x\rightarrow1\
\mbox{and}\
\int_{\Omega_b}|w_b|^4\mathrm{d}x=\varepsilon_b^2\int_\Omega u_b^4\mathrm{d}x\to\frac{2}{\beta^*}
\ \mbox{as}\ b\searrow0.
\end{equation}
%Since $u_b$ is a nonnegative minimizer for $e(b)$, $u_b$ satisfies the Euler-Lagrange equation
%$$
%\left\{\begin{array}{ll}
%-(1+b\int_{\Omega}|\nabla u_b|^2\mathrm{d}x)\Delta u_b+V(x)u_b=\mu_b u_b+\beta u_b^3  &\mbox{in}\ {\Omega}, \\[0.1cm]
% u_b=0&\mbox{on}\ {\partial\Omega}, \\[0.1cm]
%\end{array}
%\right.
%$$
%where $\mu_b\in\mathbb R$ is the associated Lagrange multiplier. Clearly,
Recall that
$$
\mu_b=\int_{\Omega}|\nabla u_b|^2\mathrm{d}x+\int_\Omega V(x)u_b^2\mathrm{d}x
+b\Big(\int_\Omega|\nabla u_b|^2\mathrm{d}x\Big)^2-\beta\int_\Omega u_b^4\mathrm{d}x,
$$
which, together with \eqref{eqn:int-V-u-b-0} and \eqref{eqn:varepsilon-2-u-b-2-u-b-4-convergence}, yields that
\begin{equation*}\label{eqn:mu-b-estimate}
\varepsilon_b^2u_b\to-\frac{\beta}{\beta^*}\ \mbox{as}\ b\searrow0.
\end{equation*}
Note that $w_b$ satisfies the following equation
\begin{equation*}\label{eqn:w-b-equation}
\left\{\begin{array}{ll}
-(1+b\varepsilon_b^{-2}\int_{\Omega_b}|\nabla w_b|^2\mathrm{d}x)\Delta w_b+\varepsilon_b^2V(\varepsilon_bx+z_b)w_b=\varepsilon_b^2\mu_bw_b+\beta w_b^3  &\mbox{in}\ {\Omega_b}, \\[0.1cm]
 w_b=0&\mbox{on}\ {\partial\Omega_b}. \\[0.1cm]
\end{array}
\right.
\end{equation*}
Repeat the proof of below \eqref{eqn:varphi-b-equation} in Lemma \ref{lem:blow-up-beta-equiv}, and replace $\epsilon_b$, $\varphi_b$, $\hat\varphi_b$ with $\varepsilon_b$, $w_b$, $\hat w_b$, respectively.
Then we can obtain \eqref{eqn:w-b-2-theta-strict-0}. Moreover, similar to Lemma \ref{lem:varphi-b-k-Q-convergence}, we accomplish the proof of ($ii$).
\end{proof}

\subsection{Mass concentration at an interior point}

In this subsection, we prove that the mass of minimizers for $e(b)$ concentrates at an interior point of $\Omega$ under $\mathcal Z_1\neq\emptyset$, i.e., Theorem \ref{thm:mass-concentration-minimizer-beta-strict} is established. In the following, the upper bound estimate of $e(b)$ as $b\searrow0$ is given.

\begin{lem}\label{lem:e-b-upper-bound}
Suppose $V(x)$ satisfies \eqref{eqn:V-potential-x-x-i} and $\mathcal Z_1\neq\emptyset$, there holds
$$
e(b)\leq-\frac{1}{2b}\Big(\frac{\beta-\beta^*}{\beta^*}\Big)^2+
\frac{\kappa\varepsilon_b^p}{\beta^*}\int_{\mathbb R^2}|x|^pQ^2\mathrm{d}x\ \mbox{as}\ b\searrow0.
$$
\end{lem}

\begin{proof}
Since $\mathcal Z_1\neq\emptyset$ and $\Omega$ is a bounded domain of $\mathbb R^2$, then there exists an open ball $B_{2R}(x_i)\subset\Omega$ centered at an inner point $x_i\in\mathcal Z_1$, where $R>0$ is sufficiently small.
Moreover, there exists $M>0$ such that
$$
V(x)\leq M|x-x_i|^p\ \mbox{in}\ \Omega,
$$
which indicates that
\begin{equation}\label{eqn:V-leq-M-p-x-x-i}
V(\varepsilon_bx+x_i)\leq M\varepsilon_b^p|x|^p\ \mbox{for\ any}\ x\in\Omega_{b_i}:=\{x\in\mathbb R^2:(\varepsilon_bx+x_i)\in\Omega\}.
\end{equation}
Noting that $\varepsilon_b=r_b^{-\frac{1}{2}}$, we take the same test function as \eqref{eqn:e-b-bar-e-b-0} while $x_0$ is replaced by $x_i\in\mathcal Z_1$, i.e.,
$$
\tilde\phi_b(x)=\frac{A_b}{\varepsilon_b\sqrt{\beta^*}}\psi(x-x_i)Q(\frac{x-x_i}{\varepsilon_b}),
$$
where $A_b>0$ is determined such that $\int_{\Omega}|\tilde\phi_b|^2\mathrm{d}x=1$.
Then it follows from \eqref{eqn:V-potential-x-x-i} and \eqref{eqn:V-leq-M-p-x-x-i} that
$$
\begin{aligned}
\int_\Omega V(x)\tilde\phi_b^2\mathrm{d}x
&=\frac{A_b^2}{\varepsilon_b^{2}\beta^*}\int_{\mathbb R^2}V(x)\psi^2(x-x_i)Q^2(\frac{x-x_i}{\varepsilon_b})\mathrm{d}x\\
&=\frac{A_b^2}{\beta^*}\int_{\mathbb R^2}V(\varepsilon_bx+x_i)\psi^2(\varepsilon_bx)Q^2(x)\mathrm{d}x\\
&=\frac{A_b^2}{\beta^*}\!\!\int_{|x|\leq\varepsilon_b^{-1}R}\!\!\!V(\varepsilon_bx+x_i)Q^2(x)\mathrm{d}x
+\frac{A_b^2}{\beta^*}\!\!\int_{\varepsilon_b^{-1}R\leq|x|\leq2\varepsilon_b^{-1}R}
\!\!\!V(\varepsilon_bx+x_i)\psi^2(\varepsilon_bx)Q^2(x)\mathrm{d}x\\
&\leq\frac{A_b^2}{\beta^*}\!\!\int_{\mathbb R^2}\!\!\!V(\varepsilon_bx+x_i)Q^2(x)\mathrm{d}x
-\frac{A_b^2}{\beta^*}\!\!\int_{|x|\geq\varepsilon_b^{-1}R}\!\!\!V(\varepsilon_bx+x_i)Q^2(x)\mathrm{d}x\\
&\quad\ +\frac{A_b^2}{\beta^*}M\varepsilon_b^p\int_{\varepsilon_b^{-1}R\leq|x|\leq2\varepsilon_b^{-1}R}
|x|^pQ^2(x)\mathrm{d}x\\
&=\frac{\kappa_i\varepsilon_b^p}{\beta^*}\int_{\mathbb R^2}|x|^pQ^2\mathrm{d}x+o(\varepsilon_b^{p})
\ \mbox{as}\ b\searrow0.
\end{aligned}
$$
Thus, similar to \eqref{eqn:A-b-1}-\eqref{eqn:nabla-phi-b-nabla-bar-u-b}, we can conclude that
$$
\begin{aligned}
e(b)-\bar e(b)&\leq E_b(\tilde\phi_b)-\bar E_b(\bar u_b)\\
&=\bar E_b(\tilde\phi_b)-\bar E_b(\bar u_b)+\int_\Omega V(x)\tilde\phi_b^2\mathrm{d}x\\
&\leq\int_\Omega V(x)\tilde\phi_b^2\mathrm{d}x+o(e^{-\frac{1}{4}r_b^{\frac{1}{2}}R})\\
&\leq\frac{\kappa_i\varepsilon_b^p}{\beta^*}\int_{\mathbb R^2}|x|^pQ^2\mathrm{d}x+o(\varepsilon_b^{p})
\ \mbox{as}\ b\searrow0,\ i\in\{1,2,\cdots,n\},
\end{aligned}
$$
which, together with \eqref{eqn:bar-e-b-value}, implies that
$$
\limsup\limits_{b\searrow0}e(b)
\leq\bar e(b)+\frac{\kappa\varepsilon_b^p}{\beta^*}\int_{\mathbb R^2}|x|^pQ^2\mathrm{d}x=-\frac{1}{2b}\Big(\frac{\beta-\beta^*}{\beta^*}\Big)^2
+\frac{\kappa\varepsilon_b^p}{\beta^*}\int_{\mathbb R^2}|x|^pQ^2\mathrm{d}x.
$$
The proof is completed.
\end{proof}

\begin{lem}\label{lem:z-b-k-x-i-mathcal-Z-1}
Suppose $V(x)$ satisfies \eqref{eqn:V-potential-x-x-i} and $\mathcal Z_1\neq\emptyset$. Let $u_b$ be a nonnegative minimizer of $e(b)$ and $z_b$ be a maximum point of $u_b$ in $\Omega$, then for the subsequence $\{b_k\}$ obtained in Lemma \ref{lem:w-b-k-convergence-property}, there holds $z_{b_k}\rightarrow x_i$
for some $x_i\in\mathcal Z_1$.
\end{lem}

\begin{proof}
By Lemma \ref{lem:w-b-k-convergence-property}, we know that $z_{b_k}\to x_{i_0}$ as $k\rightarrow\infty$ with $V(x_{i_0})=0$. Now we first claim that $\{\frac{|z_{b_k}-x_{i_0}|}{\varepsilon_{b_k}}\}$ is bounded uniformly as $k\to\infty$ by contradiction. Assume that there exists a subsequence, still denoted by $b_k$, of $\{b_k\}$ such that
$\frac{|z_{b_k}-x_{i_0}|}{\varepsilon_{b_k}}\to\infty$ as $k\to\infty$, then for any large constant $M>0$, we can derive \begin{equation*}\label{eqn:bounded-V-frac-z-a-k-x-i-varepsilon}
\begin{aligned}
&\quad\liminf\limits_{\varepsilon_{b_k}\rightarrow0}\frac{1}{\varepsilon_{b_k}^{p_{i_0}}}
\int_{\mathbb R^2}V(\varepsilon_{b_k}x+z_{b_k})|w_{b_k}|^2\mathrm{d}x\\
&\geq C\int_{B_{2R}(0)\cap\Omega_{b_k}}\liminf\limits_{\varepsilon_{a_k}\rightarrow0}\Big|x
+\frac{z_{b_k}-x_{i_0}}{\varepsilon_{b_k}}\Big|^{p_{i_0}}
\prod\limits_{j=1,j\neq i_0}^n|\varepsilon_{b_k}x+z_{b_k}-x_j|^{p_j}|w_{b_k}|^2\mathrm{d}x\geq M.
\end{aligned}
\end{equation*}
Hence, combining \eqref{eqn:GN-type-inequality}, \eqref{eqn:bar-u-b-form-r-b-form} and \eqref{eqn:varepsilon-2-u-b-2-u-b-4-convergence}, we can conclude that
$$
\begin{aligned}
e(b_k)&\geq\frac{\beta^*-\beta}{\beta^*}\int_\Omega|\nabla u_{b_k}|^2\mathrm{d}x
+\int_{\Omega}V(x)u_{b_k}^2\mathrm{d}x+\frac{b_k}{2}\Big(\int_\Omega|\nabla u_{b_k}|^2\mathrm{d}x\Big)^2\\
&=-\frac{\beta-\beta^*}{\beta^*\varepsilon_{b_k}^2}\int_{\Omega_{b_k}}|\nabla w_{b_k}|^2\mathrm{d}x
+\int_{\Omega_{b_k}}V(\varepsilon_{b_k}x+z_{b_k})w_{b_k}^2\mathrm{d}x
+\frac{\beta-\beta^*}{2\beta^*\varepsilon_{b_k}^2}\Big(\int_{\Omega_{b_k}}|\nabla w_{b_k}|^2\mathrm{d}x\Big)^2\\
&\geq-\frac{\beta-\beta^*}{2\beta^*\varepsilon_{b_k}^2}+M\varepsilon_{b_k}^{p}
=-\frac{1}{2b_k}\Big(\frac{\beta-\beta^*}{\beta^*}\Big)^2+M\varepsilon_{b_k}^{p}\ \mbox{as}\ k\to\infty,
\end{aligned}
$$
where we have used the inequality $-s+\frac{1}{2}s^2\geq-\frac{1}{2}$ for $s\in\mathbb R$. This contradicts with Lemma \ref{lem:e-b-upper-bound} and the above claim is proved.

Now, we show that $z_{b_k}\to x_i$ as $k\to\infty$ for some $x_i\in\Omega$. If not, we see that Lemma \ref{lem:similar-to-Guo-Prop-3.1-beta-equiv} also holds. That is, there exists a sequence $\{\rho_k\}$ with $\rho_k\to0$ as $k\to\infty$ such that
\begin{equation}\label{eqn:important-boundary}
\frac{\int_{\Omega_{b_k}}|w_{b_k}|^{4}\mathrm{d}x}
{\int_{\Omega_{b_k}}|\nabla w_{b_k}|^2\mathrm{d}x}
\leq\frac{2}{\beta^*}-Ce^{-\frac{2}{1+\rho_k}\frac{|z_{b_k}-x_i|}{\varepsilon_{b_k}}}
\ \mbox{as}\ k\rightarrow\infty.
\end{equation}
Then it follows from the above claim, $\varepsilon_{b_k}=r_{b_k}^{-\frac{1}{2}}\rightarrow0$ as $k\to\infty$ and \eqref{eqn:important-boundary} that
$$
\begin{aligned}
e(b_k)=E_b(u_{b_k})&\geq\frac{1}{\varepsilon_{b_k}^2}\int_{\Omega_{b_k}}|\nabla w_{b_k}|^2\mathrm{d}x
-\frac{b_k}{2\varepsilon_{b_k}^4}\Big(\int_{\Omega_{b_k}}|\nabla w_{b_k}|^2\mathrm{d}x\Big)^2
-\frac{\beta}{2\varepsilon_{b_k}^2}\int_{\Omega_{b_k}}|w_{b_k}|^4\mathrm{d}x\\
&\geq\frac{\beta^*-\beta}{\varepsilon_{b_k}^2\beta^*}\int_{\Omega_{b_k}}|\nabla w_{b_k}|^2\mathrm{d}x
+C\frac{1}{\varepsilon_{b_k}^2}e^{-\frac{2}{1+\rho_k}\cdot\frac{|z_{b_k}-x_i|}{\varepsilon_{b_k}}}
\int_{\Omega_{b_k}}|\nabla w_{b_k}|^2\mathrm{d}x\\
&\quad\ -\frac{\beta^*-\beta}{2\beta^*\varepsilon_{b_k}^2}\Big(\int_{\Omega_{b_k}}|\nabla w_{b_k}|^2\mathrm{d}x\Big)^2\\
&\geq-\frac{\beta-\beta^*}{2\beta^*\varepsilon_{b_k}^2}
+C\frac{1}{\varepsilon_{b_k}^2}e^{-\frac{2}{1+\rho_k}\cdot\frac{|z_{b_k}-x_i|}{\varepsilon_{b_k}}}(1+o(1))\\
&=-\frac{1}{2b_k}\Big(\frac{\beta-\beta^*}{\beta^*}\Big)^2
+C\frac{1}{\varepsilon_{b_k}^2}e^{-\frac{2}{1+\rho_k}\cdot\frac{|z_{b_k}-x_i|}{\varepsilon_{b_k}}}(1+o(1))
\ \mbox{as}\ k\to\infty,
\end{aligned}
$$
contradicting Lemma \ref{lem:e-b-upper-bound}. Hence the desired result is obtained. Then
$$
\begin{aligned}
\liminf\limits_{k\to\infty}\frac{e(b_k)-\bar e(b_k)}{\varepsilon_{b_k}^p}
&\geq\liminf\limits_{k\to\infty}\frac{E_b(u_{b_k})-\bar E_b(u_{b_k})}{\varepsilon_{b_k}^p}\\
&=\liminf\limits_{k\to\infty}\frac{1}{\varepsilon_{b_k}^p}
\int_{\Omega_{b_k}}V(\varepsilon_{b_k}x+z_{b_k})w_{b_k}^2\mathrm{d}x\\
&=\liminf\limits_{k\to\infty}\frac{1}{\varepsilon_{b_k}^{p-p_{i_0}}}
\int_{\Omega_{b_k}}\frac{V(\varepsilon_{b_k}x+z_{b_k})}{|\varepsilon_{b_k}x+z_{b_k}-x_{i_0}|^{p_{i_0}}}
\Big|x+\frac{z_{b_k}-x_{i_0}}{\varepsilon_{b_k}}\Big|^{p_{i_0}}w_{b_k}^2\mathrm{d}x.
\end{aligned}
$$
Combining \eqref{eqn:w-b-2-theta-strict-0} and Lemma \ref{lem:e-b-upper-bound}, we can deduce that $p_{i_0}=p$.
The proof is completed.
\end{proof}

\noindent\textbf{Proof of Theorem \ref{thm:mass-concentration-minimizer-beta-strict}.}
According to Lemma \ref{lem:z-b-k-x-i-mathcal-Z-1}, we see that $\{\frac{|z_{b_k}-x_{i}|}{\varepsilon_{b_k}}\}$
is bounded uniformly as $k\to\infty$. Thus we assume that $\frac{z_{b_k}-x_i}{\varepsilon_{b_k}}\to y_0$ as $k\to\infty$ for some $y_0\in\mathbb R^2$ with $x_i\in\mathcal Z_1$. Hence $\Omega_0=\lim_{k\to\infty}\Omega_{b_k}=\mathbb R^2$ and \eqref{eqn:beta-strict-w-b-k-convergence-Q} holds. Applying Fatou's lemma, we get
$$
\liminf\limits_{k\to\infty}\frac{e(b_k)-\bar e(b_k)}{\varepsilon_{b_k}^p}
\geq\frac{\kappa_i}{\beta^*}\int_{\mathbb R^2}|x+y_0|^pQ^2\mathrm{d}x
\geq\frac{\kappa_i}{\beta^*}\int_{\mathbb R^2}|x|^pQ^2\mathrm{d}x
\geq\frac{\kappa}{\beta^*}\int_{\mathbb R^2}|x|^pQ^2\mathrm{d}x,
$$
which, together with Lemma \ref{lem:e-b-upper-bound}, implies
$$
\lim\limits_{k\to\infty}\frac{e(b_k)-\bar e(b_k)}{\varepsilon_{b_k}^p}
=\frac{\kappa}{\beta^*}\int_{\mathbb R^2}|x|^pQ^2\mathrm{d}x
=\frac{2\lambda^{p+2}}{p}
\ \mbox{and}\
\lim\limits_{k\to\infty}\frac{|z_{b_k}-x_i|}{\varepsilon_{b_k}}=0.
$$
The proof of Theorem \ref{thm:mass-concentration-minimizer-beta-strict} is complete.
\qed

\subsection{Mass concentration near the boundary}

In this subsection, we establish Theorem \ref{thm:mass-concentration-minimizer-beta-strict-boundary}, which is concerned with the boundary mass concentration of nonnegative minimizers under the assumption that $\mathcal Z_1=\emptyset$. We begin with the following very delicate estimate.

\begin{lem}\label{lem:upper-bound-beta-strict-geq-boundary}
Suppose that $V(x)$ satisfies \eqref{eqn:V-potential-x-x-i} and $\mathcal Z_1=\emptyset$, there holds
$$
e(b)\leq-\frac{1}{2b}\Big(\frac{\beta-\beta^*}{\beta^*}\Big)^2
+\kappa\Big(\frac{p+2}{2}\Big)^p\Big(\frac{\ln\tau_b}{\tau_b}\Big)^p\ \mbox{as}\ b\searrow0,
$$
where $\tau_b=r_b^{\frac{1}{2}}\to\infty$ as $b\searrow0$.
\end{lem}

\begin{proof}
The same as the setting and notation in the proof of Lemma \ref{lem:beta-equiv-boundary-energy-estimate},
%From $\mathcal Z_1=\emptyset$, we know $\mathcal Z_0\neq\emptyset$. Then we can choose $x_i\in\mathcal Z_0$ such that $p_i=p$. From the interior ball conditions of $\Omega$, there exists an open ball $B_R(x_0)$ such that $x_i\in\partial\Omega\cap\partial B_R(x_0)$, where $R>0$ is sufficiently small. Define $R_\tau:=\frac{f(\tau)}{\tau}<R$ for $\tau>0$ large enough, where $0<f(\tau)\in C^2(\mathbb R)$ is chosen to satisfy $\lim_{\tau\rightarrow\infty}f(\tau)=\infty$ and $\lim_{\tau\rightarrow\infty}\frac{f(\tau)}{\tau}=0$. Let $x_\tau:=x_i-(1+\xi(\tau))R_\tau\overrightarrow{n}$, where
%$\overrightarrow{n}$ is the unit outward normal vector to $\partial\Omega$ at the point $x_i$ and $\xi(\tau)>0$ is determined such that $\xi(\tau)\rightarrow0$ as $\tau\rightarrow\infty$. Clearly, $B_{(1+\xi(\tau))R_\tau}(x_\tau)\subset\Omega$ and $x_i\in\partial\Omega\cap\partial B_{(1+\eta(\tau))R_\tau}(x_\tau)$ satisfying $\lim_{\tau\rightarrow\infty}x_\tau=x_i$.
%Let $\psi_\tau(x)\in C_0^\infty(\mathbb R^2)$ be a nonnegative smooth cut-off function satisfying $\psi_\tau(x)=1$ for $|x|\leq 1$ and $\psi_\tau(x)=0$ for $|x|\geq 1+\xi(\tau)$. Without loss of generality, we assume that $|\nabla\psi_\tau(x)|\leq\frac{M}{\xi(\tau)}$, where $M>0$ is independent of $\tau$.
we also using the test function
$$
\Psi_\tau(x)=\frac{A_\tau\tau}{\beta^*}\psi_\tau\Big(\frac{x-x_\tau}{R_\tau}\Big)Q(\tau(x-x_\tau)),\ x\in\Omega,
$$
where $A_\tau>0$ is chosen such that $\int_\Omega|\Psi_\tau(x)|^2\mathrm{d}x=1$. Recall that
$\Omega_\tau:=\{x\in\mathbb R^2:\big(\frac{x}{\tau}+x_\tau\big)\in\Omega\}$. The definition of $x_\tau$ indicates that $\Omega_\tau\rightarrow\mathbb R^2$ as $\tau\rightarrow\infty$.
Similar to the proof of Lemma \ref{lem:beta-equiv-boundary-energy-estimate}, we have
%$$
%\begin{aligned}
%\int_{\Omega}|\nabla\Psi_\tau|^2\mathrm{d}x
%&\leq\int_{\mathbb R^2}\frac{A_\tau^2\tau^2}{\beta^*}|\nabla Q|^2\mathrm{d}x
%+C\Big(\frac{\tau}{R_\tau\xi(\tau)}+\frac{1}{\xi^2(\tau)R_\tau^2}\Big)e^{-2\tau R_\tau}\\
%&=\int_{\mathbb R^2}|\nabla\bar u_b|^2\mathrm{d}x+C\tau^2e^{-2\tau R_\tau}\ \mbox{as}\ \tau\rightarrow\infty,
%\end{aligned}
%$$
%and
%$$
%\begin{aligned}
%\int_\Omega|\Psi_\tau(x)|^4\mathrm{d}x
%&\geq\int_{\mathbb R^2}\!\!\frac{\tau^2}{(\beta^*)^2}Q^4(x)\mathrm{d}x-o(\tau^2e^{-2\tau R_\tau})
%=\int_{\mathbb R^2}\bar u_b^4\mathrm{d}x-o(\tau^2e^{-2\tau R_\tau})\ \mbox{as}\ \tau\rightarrow\infty.
%\end{aligned}
%$$
$$
\begin{aligned}
\int_{\Omega}|\nabla\Psi_\tau|^2\mathrm{d}x
&\leq\int_{\mathbb R^2}\frac{A_\tau^2\tau^2}{\beta^*}|\nabla Q|^2\mathrm{d}x
+C\Big(\frac{\tau}{R_\tau\xi(\tau)}+\frac{1}{\xi^2(\tau)R_\tau^2}\Big)e^{-2\tau R_\tau}\ \mbox{as}\ \tau\rightarrow\infty
\end{aligned}
$$
and
$$
\begin{aligned}
\int_\Omega|\Psi_\tau(x)|^4\mathrm{d}x
&\geq\int_{\mathbb R^2}\!\!\frac{\tau^2}{(\beta^*)^2}Q^4(x)\mathrm{d}x-o(\tau^2e^{-2\tau R_\tau})
\ \mbox{as}\ \tau\rightarrow\infty.
\end{aligned}
$$
Now, we choose a suitable function $\xi(\tau)>0$ so that $\frac{\tau}{R_\tau\xi(\tau)}=o(\tau^2)$, $\frac{1}{R_\tau^2\xi^2(\tau)}=o(\tau^2)$ as $\tau\rightarrow\infty$ and $f(\tau)=\tau R_\tau=\frac{p+2}{2}\ln\tau$, then we get
$C\tau^2e^{-2f(\tau)}=C\tau^{-p}=o\Big(\Big(\frac{\ln\tau}{\tau}\Big)^p\Big)\ \mbox{as}\ \tau\rightarrow\infty$.
Thus, we derive
$$
\begin{aligned}
\int_{\Omega}V(x)|\Psi_\tau(x)|^2\mathrm{d}x
&=\kappa_i\Big(\frac{p+2}{2}\Big)^p\Big(\frac{\ln\tau}{\tau}\Big)^p+o\Big(\Big(\frac{\ln\tau}{\tau}\Big)^p\Big)
\ \mbox{as}\ \tau\rightarrow\infty.
\end{aligned}
$$
In particular, letting $\tau=\tau_b:=r_b^{\frac{1}{2}}$,
it results that
$$
\begin{aligned}
&\quad\ \int_{\mathbb R^2}\frac{A_\tau^2\tau^2}{\beta^*}|\nabla Q|^2\mathrm{d}x
+C\Big(\frac{\tau}{R_\tau\xi(\tau)}+\frac{1}{\xi^2(\tau)R_\tau^2}\Big)e^{-2\tau R_\tau}\\
&=\int_{\mathbb R^2}|\nabla\bar u_b|^2\mathrm{d}x+Cr_be^{-2r_b^{\frac{1}{2}}R_{\tau_b}}
=\int_{\mathbb R^2}|\nabla\bar u_b|^2\mathrm{d}x+o\Big(\Big(\frac{\ln\tau_b}{\tau_b}\Big)^p\Big)\ \mbox{as}\ b\searrow0
\end{aligned}
$$
and
$$
\int_{\mathbb R^2}\!\!\frac{\tau^2}{(\beta^*)^2}Q^4(x)\mathrm{d}x-o(\tau^2e^{-2\tau R_\tau})
=\int_{\mathbb R^2}\bar u_b^4\mathrm{d}x-o\Big(\Big(\frac{\ln\tau_b}{\tau_b}\Big)^p\Big)\ \mbox{as}\ \ b\searrow0.
$$
Hence from above estimates,
$$
\begin{aligned}
e(b)-\bar e(b)&\leq E_b(\Psi_\tau)-\bar E_b(\bar u_b)\\
&\leq \kappa_i\Big(\frac{p+2}{2}\Big)^p\Big(\frac{\ln\tau_b}{\tau_b}\Big)^p+o\Big(\Big(\frac{\ln\tau_b}{\tau_b}\Big)^p\Big)
\ \mbox{as}\ b\searrow0,\ i\in\{1,2,\cdots,n\}.
\end{aligned}
$$
The proof is completed.
\end{proof}

\begin{lem}
Assume that $V(x)$ satisfies \eqref{eqn:V-potential-x-x-i} and $\mathcal Z_1=\emptyset$. Let $\{b_k\}$ be the convergent subsequence given by Lemma \ref{lem:w-b-k-convergence-property}-(ii), then $z_{b_k}\to x_i$ as $k\to\infty$ for some $x_i\in\mathcal Z_0$.
\end{lem}

\begin{proof}
In view of Lemma \ref{lem:w-b-k-convergence-property}, we see $z_{b_k}\to x_{i_0}$ as $k\to\infty$ with $V(x_{i_0})=0$.
We first consider $x_{i_0}\in\Omega$ with $0<p_{i_0}<p$, or $x_{i_0}\in\partial\Omega$ satisfying $0<p_{i_0}<p$ and $\frac{|z_{b_k}-x_{i_0}|}{\varepsilon_{b_k}}\rightarrow\infty$ as $k\rightarrow\infty$. In this case, it results that $\Omega_0=\lim_{k\rightarrow\infty}\Omega_{b_k}=\mathbb R^2$.
Then by \eqref{eqn:w-b-2-theta-strict-0}, we get that for sufficiently small $R>0$,
there exists $C_0(R)>0$, independent of $b_k$, such that
$$
\begin{aligned}
&\quad\liminf\limits_{\varepsilon_{b_k}\rightarrow0}\frac{1}{\epsilon_{b_k}^{p_{i_0}}}
\int_{\mathbb R^2}V(\varepsilon_{b_k}x+z_{b_k})|w_{b_k}|^2\mathrm{d}x\\
&\geq C\int_{B_{R}(0)}\liminf\limits_{\varepsilon_{b_k}\rightarrow0}\Big|x
+\frac{z_{b_k}-x_{i_0}}{\varepsilon_{b_k}}\Big|^{p_{i_0}}
\prod\limits_{j=1,j\neq i_0}^n|\varepsilon_{b_k}x+z_{b_k}-x_j|^{p_j}|w_{b_k}|^2\mathrm{d}x
\geq C_0(R).
\end{aligned}
$$
Note that $\varepsilon_{b_k}=r_{b_k}^{-\frac{1}{2}}$ and $\tau_{b_k}=r_{b_k}^{\frac{1}{2}}$. Thus,
$$
\liminf\limits_{k\to\infty}\frac{e(b_k)-\bar e(b_k)}{(\frac{\ln\tau_{b_k}}{\tau_{b_k}})^p}
\geq\liminf\limits_{k\to\infty}\frac{\int_{\Omega_{b_k}}V(\varepsilon_{b_k}x+z_{b_k})w_{b_k}^2\mathrm{d}x}
{(\frac{\ln\tau_{b_k}}{\tau_{b_k}})^p}
\geq C_0(R)\tau_{b_k}^{p-p_{i_0}}\Big(\frac{1}{\ln\tau_{b_k}}\Big)^p,
$$
which is in contrast with Lemma \ref{lem:upper-bound-beta-strict-geq-boundary}.

On the other hand, if $x_{i_0}\in\partial\Omega$ satisfies $0<p_{i_0}<p$ and $\{\frac{|z_{b_k}-x_{i_0}|}{\varepsilon_{b_k}}\}$ is uniformly bounded as $k\rightarrow\infty$, then there exists a sequence $\{\rho_k\}$ with $\rho_k\to0$ as $k\to\infty$ such that \eqref{eqn:important-boundary} holds.
It follows that
\begin{equation}\label{eqn:beta-strict-interior-rule-out-second-case}
\begin{aligned}
e(b_k)=E_b(u_{b_k})&\geq-\frac{\beta-\beta^*}{2\beta^*\varepsilon_{b_k}^2}
+C\frac{1}{\varepsilon_{b_k}^2}e^{-\frac{2}{1+\rho_k}\cdot\frac{|z_{b_k}-x_i|}{\varepsilon_{b_k}}}(1+o(1))\\
&=-\frac{1}{2b_k}\Big(\frac{\beta-\beta^*}{\beta^*}\Big)^2
+C\frac{1}{\varepsilon_{b_k}^2}e^{-\frac{2}{1+\rho_k}\cdot\frac{|z_{b_k}-x_i|}{\varepsilon_{b_k}}}(1+o(1))
\ \mbox{as}\ k\to\infty,
\end{aligned}
\end{equation}
which contradicts Lemma \ref{lem:upper-bound-beta-strict-geq-boundary}.
Therefore, we conclude that $z_{b_k}\to x_{i_0}$ as $k\to\infty$ with $x_{i_0}\in\partial\Omega$ and $p_{i_0}=p$. The proof is completed.
\end{proof}

\begin{lem}\label{lem:beta-strict-boundary-z-b-k-x-i-varepsilon-b-k-ln}
Under the assumption \eqref{eqn:V-potential-x-x-i} and $Z_1=\emptyset$, we have
$$
\limsup\limits_{k\rightarrow\infty}\frac{|z_{b_k}-x_i|}{\varepsilon_{b_k}|\ln\varepsilon_{b_k}|}<\infty\ \mbox{and}\
\liminf\limits_{k\rightarrow\infty}\frac{|z_{b_k}-x_i|}{\varepsilon_{b_k}|\ln\varepsilon_{b_k}|}\geq\frac{p+2}{2}.
$$
\end{lem}

\begin{proof}
We argue the first one by contradiction. Assume that there exists a subsequence of $\{b_k\}$ such that $\frac{|z_{b_k}-x_i|}{\varepsilon_{b_k}|\ln\varepsilon_{b_k}|}\rightarrow\infty$ as $k\rightarrow\infty$, then we can derive from \eqref{eqn:w-b-2-theta-strict-0} that for any large constant $M>0$,
$$
\begin{aligned}
&\quad\liminf\limits_{\varepsilon_{b_k}\rightarrow0}\frac{1}{(\varepsilon_{b_k}|\ln\varepsilon_{b_k}|)^p}
\int_{B_{2R}(0)}V(\varepsilon_{b_k}x+z_{b_k})|w_{b_k}|^2\mathrm{d}x\\
&\geq C\int_{B_{2R}(0)\cap\Omega_{b_k}}\liminf\limits_{\varepsilon_{b_k}\rightarrow0}
\Big|\frac{x}{|\ln\epsilon_{b_k}|}+\frac{z_{b_k}-x_{i}}{\varepsilon_{b_k}|\ln\varepsilon_{b_k}|}\Big|^{p}
\prod\limits_{j=1,j\neq i}^n|\varepsilon_{b_k}x+z_{b_k}-x_j|^{p_j}|w_{b_k}|^2\mathrm{d}x\\
&\geq M,
\end{aligned}
$$
which implies that
$$
\frac{e(b_k)-\bar e(b_k)}{(\frac{\ln\tau_{b_k}}{\tau_{b_k}})^p}
\geq\frac{\int_{\Omega_b}V(\varepsilon_{b_k}x+z_{b_k})w_{b_k}^2\mathrm{d}x}
{(\frac{\ln\tau_{b_k}}{\tau_{b_k}})^p}
\geq M\ \mbox{as}\ k\to\infty.
$$
Namely,
$$
e(b_k)-\bar e(b_k)\geq M\Big(\frac{\ln\tau_{b_k}}{\tau_{b_k}}\Big)^p\ \mbox{as}\ k\to\infty,
$$
contradicting with Lemma \ref{lem:upper-bound-beta-strict-geq-boundary}.

For the remain one, we suppose that there is a subsequence of $\{b_k\}$ such that
$\lim_{k\to\infty}\frac{|z_{b_k}-x_i|}{\varepsilon_{b_k}|\ln\varepsilon_{b_k}|}=\nu<\frac{p+2}{2}$. For sufficiently large $k$, there exists a sufficiently large $k_0>0$ such that
$$
\frac{2\nu}{1+\rho_k}-p-2<\frac{1}{2}(2\nu-p-2)\ \mbox{for}\ k>k_0.
$$
Similar to
\eqref{eqn:beta-strict-interior-rule-out-second-case}, we can get
$$
\begin{aligned}
e(b_k)=E_b(u_{b_k})
&\geq-\frac{1}{2b_k}\Big(\frac{\beta-\beta^*}{\beta^*}\Big)^2
+C\frac{1}{\varepsilon_{b_k}^2}e^{-\frac{2}{1+\rho_k}\cdot\frac{|z_{b_k}-x_i|}{\varepsilon_{b_k}}}(1+o(1))\\
&=-\frac{1}{2b_k}\Big(\frac{\beta-\beta^*}{\beta^*}\Big)^2+C\varepsilon_{b_k}^{\frac{2\nu}{1+\rho_k}-2}(1+o(1))\\
&>-\frac{1}{2b_k}\Big(\frac{\beta-\beta^*}{\beta^*}\Big)^2
+C\varepsilon_{b_k}^p\cdot\varepsilon_{b_k}^{\frac{1}{2}(2\nu-p-2)}
\ \mbox{as}\ k\to\infty,
\end{aligned}
$$
which also contradicts Lemma \ref{lem:upper-bound-beta-strict-geq-boundary}. Hence the proof is completed.
\end{proof}

\noindent\textbf{Proof of Theorem \ref{thm:mass-concentration-minimizer-beta-strict-boundary}.}
Due to Lemma \ref{lem:beta-strict-boundary-z-b-k-x-i-varepsilon-b-k-ln}, we have $\frac{|z_{b_k}-x_i|}{\varepsilon_{b_k}}\to\infty$ as $k\to\infty$, which states $\Omega_0=\lim_{k\to\infty}\Omega_{b_k}=\mathbb R^2$. Then it follows from Lemma \ref{lem:w-b-k-convergence-property} that
$$
\lim\limits_{k\rightarrow\infty}w_{b_k}(x)=w_0=\frac{Q(|x|)}{\sqrt{\beta^*}}\ \mbox{in}\ H^1(\mathbb R^2).
$$
Hence $w_0$ admits the exponential decay as $|x|\to\infty$. By comparison principle, $w_{b_k}$ also admits the exponential decay near $|x|\to\infty$ as $k\to\infty$. In view of Lemma \ref{lem:beta-strict-boundary-z-b-k-x-i-varepsilon-b-k-ln}, it results that
\begin{equation}\label{eqn:in-order-to-V-estimate-ln-varepsilon-beta-strict}
\int_{\Omega_{b_k}}\Big|\frac{x}{|\ln\varepsilon_{a_k}|}+\frac{z_{b_k}-x_i}{\varepsilon_{b_k}|\ln\varepsilon_{b_k}|}\Big|^p
\Big|w_{b_k}^2-\frac{Q^2(x)}{\beta^*}\Big|\mathrm{d}x
\rightarrow0\ \mbox{as}\ k\rightarrow\infty.
\end{equation}
Since $\{\epsilon_{b_k}x+z_{b_k}\}$ is bounded uniformly in $k$, then there exists $M>0$ such that
\begin{equation}\label{eqn:proof-of-Theorem-1.4-potential}
V(\epsilon_{b_k}x+z_{b_k})\leq M|\epsilon_{b_k}x+z_{b_k}-x_i|^p\ \mbox{for}\ x\in\Omega_{b_k}.
\end{equation}
Based on \eqref{eqn:in-order-to-V-estimate-ln-varepsilon-beta-strict} and \eqref{eqn:proof-of-Theorem-1.4-potential}, we can estimate that
$$
\begin{aligned}
&\quad\ \Big|\int_{\Omega_{b_k}}V(\varepsilon_{b_k}x+z_{b_k})w_{b_k}^2\mathrm{d}x
-\int_{\Omega_{b_k}}V(\varepsilon_{b_k}x+z_{b_k})\frac{Q^2(x)}{\beta^*}\mathrm{d}x\Big|\\
&\leq M\int_{\Omega_{b_k}}|\varepsilon_{b_k}x+z_{b_k}-x_i|^p\Big|w_{b_k}^2-\frac{Q^2(x)}{\beta^*}\Big|\mathrm{d}x\\
&=M\varepsilon_{b_k}^p|\ln\varepsilon_{b_k}|^p
\int_{\Omega_{b_k}}\Big|\frac{x}{|\ln\varepsilon_{b_k}|}
+\frac{z_{b_k}-x_i}{\varepsilon_{b_k}|\ln\varepsilon_{b_k}|}\Big|^p
\Big|w_{b_k}^2-\frac{Q^2(x)}{\beta^*}\Big|\mathrm{d}x\\
&=o(\varepsilon_{b_k}^p|\ln\varepsilon_{b_k}|^p)\ \mbox{as}\ k\rightarrow\infty.
\end{aligned}
$$
As a result,
$$
\begin{aligned}
e(b_k)=E_b(u_{b_k})
&\geq-\frac{1}{2b_k}\Big(\frac{\beta-\beta^*}{\beta^*}\Big)^2
+\int_{\Omega_{b_k}}V(\varepsilon_{b_k}x+z_{b_k})\frac{Q^2(x)}{\beta^*}\mathrm{d}x
+o(\varepsilon_{b_k}^p|\ln\varepsilon_{b_k}|^p)\\
&=-\frac{1}{2b_k}\Big(\frac{\beta-\beta^*}{\beta^*}\Big)^2+\kappa_i\zeta^p\varepsilon_{b_k}^p|\ln\varepsilon_{b_k}|^p
+o(\varepsilon_{b_k}^p|\ln\varepsilon_{b_k}|^p)\ \mbox{as}\ k\rightarrow\infty,
\end{aligned}
$$
where $\zeta:=\lim_{k\rightarrow\infty}\frac{|z_{b_k}-x_i|}{\varepsilon_{b_k}|\ln\varepsilon_{b_k}|}\geq\frac{p+2}{2}$ due to Lemma \ref{lem:beta-strict-boundary-z-b-k-x-i-varepsilon-b-k-ln}.
Together with Lemma \ref{lem:upper-bound-beta-strict-geq-boundary}, we can deduce that
$$
\zeta\leq\frac{p+2}{2}.
$$
Hence
$$
\lim\limits_{k\to\infty}\frac{e(b_k)-\bar e(b_k)}{\varepsilon_{b_k}^p|\ln\varepsilon_{b_k}|^p}
=\kappa\left(\frac{p+2}{2}\right)^p
\ \mbox{and}\
\lim_{k\rightarrow\infty}\frac{|z_{b_k}-x_i|}{\varepsilon_{b_k}|\ln\varepsilon_{b_k}|}=\frac{p+2}{2}.
$$
The proof is completed.
\qed

\vspace{0.8cm}

\noindent\textbf{\Large Acknowledgements}

\noindent {This research is supported by National Natural Science Foundation of China [No.12371120].}
\vspace{0.8em}

\noindent\textbf{\Large Declaration}

\noindent {{\bf Conflict of interest} The authors do not have conflict of interest.}

\end{document}